\documentclass{amsart}
\usepackage[dvips,final]{graphics}
\usepackage{array}
\usepackage{arydshln}
\usepackage[makeroom]{cancel}
 \usepackage[all]{xy}
 \usepackage{url}
\usepackage{multirow, blkarray}
\usepackage{booktabs}
\usepackage{textcomp}
 \usepackage[final]{epsfig}
 \usepackage{color}
\usepackage[T1]{fontenc}      
\usepackage[english,french]{babel}
\usepackage[utf8]{inputenc}
\usepackage{blindtext}

\usepackage{amsfonts,amscd,array, mathdots, epigraph}
\usepackage{amsmath}
\usepackage{amssymb}
\usepackage{amsthm}
\usepackage{mathrsfs}
\usepackage{stmaryrd}
\usepackage{slashbox}
\usepackage{diagbox}
\usepackage{enumitem}

\usepackage{ulem}
\usepackage{tikz}
\usepackage{xcolor}
\usepackage{multicol}
\definecolor{ufogreen}{rgb}{0.24, 0.82, 0.44}

\begin{document}

\newtheorem{theorem}{Théorème}[section]
\newtheorem{theore}{Théorème}
\newtheorem{definition}[theorem]{Définition}
\newtheorem{proposition}[theorem]{Proposition}
\newtheorem{corollary}[theorem]{Corollaire}
\newtheorem*{con}{Conjecture}
\newtheorem*{remark}{Remarque}
\newtheorem*{remarks}{Remarques}
\newtheorem*{pro}{Problème}
\newtheorem*{examples}{Exemples}
\newtheorem*{example}{Exemple}
\newtheorem{lemma}[theorem]{Lemme}

\numberwithin{equation}{section}

\newcommand{\bZ}{\mathbb{Z}}
\newcommand{\bC}{\mathbb{C}}
\newcommand{\bN}{\mathbb{N}}
\newcommand{\bM}{\mathbb{N^{*}}}
\newcommand{\bR}{\mathbb{R}}

\title{$\lambda$-quiddité sur certains sous-groupes monogènes de $\bC$}

\author{Flavien Mabilat}

\date{}


\address{Laboratoire de Mathématiques de Reims,
UMR9008 CNRS et Université de Reims Champagne-Ardenne, 
U.F.R. Sciences Exactes et Naturelles 
Moulin de la Housse - BP 1039 
51687 Reims cedex 2,
France}
\email{flavien.mabilat@univ-reims.fr}

\maketitle

\selectlanguage{french}
\begin{abstract}

Au cours de l'étude des frises de Coxeter, M. Cuntz a défini la notion de $\lambda$-quiddités et a soulevé le problème de l'étude de celles-ci sur certains sous-ensembles de $\bC$. L'objectif de ce texte est de mener à bien cette étude dans le cas de quelques sous-groupes monogènes de ($\bC,+$). On s'intéressera tout particulièrement aux cas des sous-groupes monogènes engendrés par $\sqrt{k}$ et par $i\sqrt{k}$ avec $k \in \bN$.
\\
\end{abstract}

\selectlanguage{english}
\begin{abstract}

During the study of Coxeter's friezes, M. Cuntz defined the concept of $\lambda$-quiddities and gave the problem of studying them over some subsets of $\bC$. The objective of this text is to carry out this study in the case of some cyclic subgroups of ($\bC,+$). In particular we will study the case of the cyclic subgroups generated by $\sqrt{k}$ and $i\sqrt{k}$, with $k \in \bN$.
\\
\end{abstract}

\selectlanguage{french}

\thispagestyle{empty}

\textbf{Mots clés :} $\lambda$-quiddité; groupe modulaire; sous-groupe monogène
\\
\\ \indent \textbf{Classification :} 05A05
\\
\begin{flushright}
 \textit{\og L'efficacité symbolique des mots ne s'exerce jamais que dans la mesure où celui 
\\qui la subit reconnaît celui qui l'exerce comme fondé à l'exercer [...]. \fg} 
\\ Pierre Bourdieu, \textit{Ce que parler veut dire}
\end{flushright}

\section{Introduction}
\label{Intro}

Depuis leur introduction au début des années soixante-dix, les frises de Coxeter font l'objet de nombreux travaux qui ont notamment mis en lumière de multiples liens avec plusieurs autres domaines mathématiques (voir par exemple \cite{Mo1}). Lorsque l'on étudie ces dernières, on est amené à considérer l'équation suivante :
\[M_{n}(a_{1},\ldots,a_{n}):=\begin{pmatrix}
   a_{n} & -1 \\[4pt]
    1    & 0 
   \end{pmatrix}
\begin{pmatrix}
   a_{n-1} & -1 \\[4pt]
    1    & 0 
   \end{pmatrix}
   \cdots
   \begin{pmatrix}
   a_{1} & -1 \\[4pt]
    1    & 0 
    \end{pmatrix}=-Id.\]
		
\noindent En effet, les solutions de cette équation sont celles qui interviennent dans la construction des frises de Coxeter (voir \cite{BR} et \cite{CH} proposition 2.4). Cela conduit naturellement à considérer la généralisation ci-dessous :
\begin{equation}
\label{a}
\tag{$E$}
M_{n}(a_1,\ldots,a_n)=\pm Id.
\end{equation}
\noindent Notons par ailleurs que les matrices $M_{n}(a_{1},\ldots,a_{n})$ interviennent également dans de nombreuses autres situations telles que l'étude des fractions continues ou encore la résolution des équations de Sturm-Liouville discrètes (voir par exemple \cite{O}).
\\
\\ \indent Les solutions de \eqref{a} sont appelées $\lambda$-quiddités. L'étude de ces dernières repose principalement sur une notion d'irréductibilité basée sur une opération sur les $n$-uplets d'éléments de l'ensemble considéré (voir \cite{C} et la section suivante). L'objectif principal est alors d'obtenir l'ensemble des solutions irréductibles de \eqref{a} lorsque les $a_{i}$ appartiennent à un ensemble fixé. On dispose de telles descriptions pour $\bN$ (voir \cite{C} Théorème 3.1 et la section \ref{class}), $\mathbb{Z}$ (voir \cite{CH} Théorème 6.2 et la section \ref{class}) et $\bZ[\alpha]$ avec $\alpha$ un nombre complexe transcendant (voir \cite{M2} Théorème 2.7 et la section \ref{class}). Par ailleurs, V. Ovsienko a donné une procédure récursive permettant de construire toutes les $\lambda$-quiddités sur $\bM$ (voir \cite{O} Théorème 2). On dispose également d'un certain nombre d'éléments sur les solutions de \eqref{a} dans les cas où on se place sur les anneaux $\bZ/N\bZ$ (voir \cite{M1,M3,M4,M5}) ou sur les corps finis (voir \cite{Mo2}).
\\
\\ \indent On souhaite ici se placer sur d'autres sous-ensembles de $\bC$ que ceux déjà étudiés, en lien avec un problème ouvert soulevé par M. Cuntz (voir \cite{C} problème 4.1). Pour cela, on va s'attacher à résoudre \eqref{a} sur un certain nombre de sous-groupes monogènes de $(\bC,+)$, c'est-à-dire sur un certain nombre de sous-groupes additifs engendrés par un seul élément. On s'intéressera tout particulièrement aux cas des sous-groupes monogènes engendrés par $\sqrt{k}$ et par $i\sqrt{k}$ (avec $k \in \bN$). Les énoncés des deux résultats principaux sont présentés dans la section suivante et démontrés respectivement dans les sections \ref{preuve25} et \ref{preuve26} tandis que la section \ref{pre} contient des résultats préliminaires.

\section{Définitions et résultats principaux}
\label{RP}

L'objectif de cette section est double. D'une part, on va donner les définitions des notions essentielles à l'étude menée dans la suite et, d'autre part, on va énoncer les résultats principaux de ce texte. Dans tout ce qui suit, $R$ est un sous-magma de $\bC$, c'est-à-dire un sous-ensemble de $\bC$ stable par addition. De plus, si $w \in \bC$, on note $<w>$ le sous-groupe de $(\bC,+)$ engendré par $w$. En d'autres termes, on a $<w>~=\mathbb{Z}w=\{kw, k \in \bZ\}$. On débute par la définition formelle du concept de $\lambda$-quiddité.

\begin{definition}[\cite{C}, définition 2.2]
\label{21}

Soit $n \in \bM$. On dit que le $n$-uplet $(a_{1},\ldots,a_{n})$ d'éléments de $R$ est une $\lambda$-quiddité sur $R$ de taille $n$ si $(a_{1},\ldots,a_{n})$ est une solution de \eqref{a}, c'est-à-dire si $(a_{1},\ldots,a_{n})$ vérifie $M_{n}(a_{1},\ldots,a_{n})=\pm Id.$ En cas d'absence d'ambiguïté, on parlera simplement de $\lambda$-quiddité.

\end{definition}

\noindent Afin de procéder à l'étude des $\lambda$-quiddités, on donne les deux définitions suivantes :

\begin{definition}[\cite{C}, lemme 2.7]
\label{22}

Soient $(n,m) \in (\bM)^{2}$, $(a_{1},\ldots,a_{n})$ un $n$-uplet d'éléments de $R$ et $(b_{1},\ldots,b_{m})$ un $m$-uplet d'éléments de $R$. On définit l'opération suivante: \[(a_{1},\ldots,a_{n}) \oplus (b_{1},\ldots,b_{m}):= (a_{1}+b_{m},a_{2},\ldots,a_{n-1},a_{n}+b_{1},b_{2},\ldots,b_{m-1}).\] Le $(n+m-2)$-uplet ainsi obtenu est appelé la somme de $(a_{1},\ldots,a_{n})$ avec $(b_{1},\ldots,b_{m})$.

\end{definition}

\begin{examples} 

{\rm On prend ici $R=\bN$. On a :} 
\begin{itemize}
\item $(2,0,3) \oplus (1,1,0) = (2,0,4,1)$;
\item $(2,3,4) \oplus (4,1,0,8) = (10,3,8,1,0)$;
\item $(1,3,5,3) \oplus (3,2,2,5,4) = (5,3,5,6,2,2,5)$;
\item $n \geq 2$, $(a_{1},\ldots,a_{n}) \oplus (0,0) = (0,0) \oplus (a_{1},\ldots,a_{n})=(a_{1},\ldots,a_{n})$.
\end{itemize}

\end{examples}

L'opération présentée ci-dessus n'est malheureusement ni commutative ni associative (voir \cite{WZ} exemple 2.1). Cependant, celle-ci est très utile pour l'étude des solutions de l'équation \eqref{a} car elle possède la propriété suivante : si $(b_{1},\ldots,b_{m})$ est une $\lambda$-quiddité sur $R$ alors la somme $(a_{1},\ldots,a_{n}) \oplus (b_{1},\ldots,b_{m})$ est une $\lambda$-quiddité sur $R$ si et seulement si $(a_{1},\ldots,a_{n})$ est une $\lambda$-quiddité sur $R$ (voir \cite{C,WZ} et \cite{M1} proposition 3.7).

\begin{definition}[\cite{C}, définition 2.5]
\label{23}

Soient $(a_{1},\ldots,a_{n})$ et $(b_{1},\ldots,b_{n})$ deux $n$-uplets d'éléments de $R$. On dit que $(a_{1},\ldots,a_{n}) \sim (b_{1},\ldots,b_{n})$ si $(b_{1},\ldots,b_{n})$ est obtenu par permutations circulaires de $(a_{1},\ldots,a_{n})$ ou de $(a_{n},\ldots,a_{1})$.

\end{definition}

On montre que $\sim$ est une relation d'équivalence sur l'ensemble des $n$-uplets d'éléments de $R$ (voir \cite{WZ}, lemme 1.7). De plus, si un $n$-uplet d'éléments de $R$ est une $\lambda$-quiddité sur $R$ alors tout $n$-uplet d'éléments de $R$ qui lui est équivalent est aussi une $\lambda$-quiddité sur $R$ (voir \cite{C} proposition 2.6). On peut maintenant définir la notion d'irréductibilité annoncée dans l'introduction.

\begin{definition}[\cite{C}, définition 2.9]
\label{24}

Une $\lambda$-quiddité $(c_{1},\ldots,c_{n})$ sur $R$ avec $n \geq 3$ est dite réductible s'il existe une $\lambda$-quiddité $(b_{1},\ldots,b_{l})$ sur $R$ et un $m$-uplet $(a_{1},\ldots,a_{m})$ d'éléments de $R$ tels que \begin{itemize}
\item $(c_{1},\ldots,c_{n}) \sim (a_{1},\ldots,a_{m}) \oplus (b_{1},\ldots,b_{l})$,
\item $m \geq 3$ et $l \geq 3$.
\end{itemize}
Une $\lambda$-quiddité est dite irréductible si elle n'est pas réductible.
\end{definition}

\begin{remark} 

{\rm Si $0 \in R$ alors $(0,0)$ est une $\lambda$-quiddité sur $R$. En revanche, elle n'est jamais considérée comme étant une $\lambda$-quiddité irréductible.}

\end{remark}

On dispose déjà de plusieurs résultats de classification des $\lambda$-quiddités irréductibles (voir section \ref{class}). Notre objectif ici est d'étudier les solutions de \eqref{a} sur certains sous-groupes monogènes de $(\bC,+$) et d'obtenir, si possible, une classification des $\lambda$-quiddités irréductibles sur ces ensembles. On commencera par le cas de $R=<\sqrt{k}>$ en démontrant le résultat ci-dessous :

\begin{theorem}
\label{25}

Soit $k \in \bN$.
\\
\\i) Si $k=0$. $(0,0,0,0)$ est la seule $\lambda$-quiddité irréductible sur $<\sqrt{k}>$.
\\
\\ii) Si $k=1$. Les $\lambda$-quiddités irréductibles sur $<\sqrt{1}>$ sont : \[\{(1,1,1), (-1,-1,-1), (0,a,0,-a), (a,0,-a,0); a \in \bZ-\{\pm 1\} \}.\]

\noindent iii) Si $k=2$. Les $\lambda$-quiddités irréductibles sur $<\sqrt{2}>$ sont : \[\{(\sqrt{2},\sqrt{2},\sqrt{2},\sqrt{2}), (-\sqrt{2},-\sqrt{2},-\sqrt{2},-\sqrt{2}), (0,a\sqrt{2},0,-a\sqrt{2}), (a\sqrt{2},0,-a\sqrt{2},0); a \in \bZ\}.\]

\noindent iv) Si $k=3$. Les $\lambda$-quiddités irréductibles sur $<\sqrt{3}>$ sont : \[\{(\sqrt{3},\sqrt{3},\sqrt{3},\sqrt{3},\sqrt{3},\sqrt{3}), (-\sqrt{3},-\sqrt{3},-\sqrt{3},-\sqrt{3},-\sqrt{3},-\sqrt{3}), (0,a\sqrt{3},0,-a\sqrt{3}), (a\sqrt{3},0,-a\sqrt{3},0); a \in \bZ\}.\]

\noindent v) Si $k \geq 4$. Les $\lambda$-quiddités irréductibles sur $<\sqrt{k}>$ sont : \[\{(0,a\sqrt{k},0,-a\sqrt{k}), (a\sqrt{k},0,-a\sqrt{k},0); a \in \bZ\}.\]

\end{theorem}

Ce théorème est démontré dans la section \ref{preuve25}. Par ailleurs, on constate que les listes de $\lambda$-quiddités irréductibles sur les sous-groupes $<\sqrt{k}>$ ($k \in \bM$) diffèrent en fonction des valeurs de $k$ alors que ces sous-groupes sont tous isomorphes. On considérera ensuite les cas $R=<i\sqrt{k}>$, avec $k \in \bN$. En particulier, on démontrera le résultat suivant :

\begin{theorem}
\label{26}

Soient $k \in \bN$ et $\Omega_{k}$ l'ensemble des $\lambda$-quiddités sur $<i\sqrt{k}>$. Si $k \neq 1$, on note $\Xi_{k}$ l'ensemble des $\lambda$-quiddités sur $<\sqrt{k}>$ et on note $\Xi_{1}$ l'ensemble des $\lambda$-quiddités de taille paire sur $<\sqrt{1}>=\mathbb{Z}$. 
\\
\\i) $\Omega_{k}$ et $\Xi_{k}$ ne contiennent que des éléments de taille paire. De plus, l'application \[\begin{array}{ccccc} 
\varphi_{k} & : & \Omega_{k} & \longrightarrow & \Xi_{k} \\
 & & (a_{1}i,a_{2}i,\ldots,a_{2n-1}i,a_{2n}i) & \longmapsto & (a_{1},-a_{2},\ldots,a_{2n-1},-a_{2n})  \\
\end{array}\] est une bijection.
\\
\\ii) Si $k \in \bN$, $k \neq 1$ alors $\varphi_{k}$ établit une bijection entre l'ensemble des $\lambda$-quiddités irréductibles sur $<i\sqrt{k}>$ et l'ensemble des $\lambda$-quiddités irréductibles sur $<\sqrt{k}>$.

\end{theorem}

\noindent Ce résultat est prouvé dans la section \ref{preuve26}.

\section{Résultats préliminaires et premiers éléments sur le cas des sous-groupes monogènes}
\label{pre}

L'objectif de cette section est de regrouper plusieurs résultats déjà connus sur l'étude des $\lambda$-quiddités et de donner quelques premiers éléments concernant le cas des sous-groupes monogènes.

\subsection{Résultats préliminaires}
\label{Rp}

On commence par rappeler les solutions de \eqref{a} sur $\bC$ pour les petites valeurs de $n$.

\begin{proposition}[\cite{CH}, exemple 2.7]
\label{31}
\begin{itemize}
\item \eqref{a} n'a pas de solution de taille 1.
\item $(0,0)$ est la seule solution de \eqref{a} de taille 2.
\item $(1,1,1)$ et $(-1,-1,-1)$ sont les seules solutions de \eqref{a} de taille 3.
\item Les solutions de \eqref{a} pour $n=4$ sont les 4-uplets suivants $(-a,b,a,-b)$ avec $ab=0$ (c'est-à-dire $a=0$ ou $b=0$) et $(a,b,a,b)$ avec $ab=2$.
\end{itemize}
 
\end{proposition}

Si l'on se place sur un sous-magma $R$ de $\bC$ alors il suffit d'adapter le résultat précédent aux éléments présents dans $R$. Par exemple, si $1,-1 \notin R$ alors il n'y a pas de $\lambda$-quiddité de taille 3 sur $R$.

\begin{proposition}[\cite{M4}, proposition 3.5]
\label{opposé}

Soit $R$ un sous-magma de $\bC$ tel que si $x \in R$ alors $-x \in R$. Soit $n$ un entier naturel supérieur à 2. $(a_{1},\ldots,a_{n})$ est une solution (irréductible) de \eqref{a} sur $R$ si et seulement si $(-a_{1},\ldots,-a_{n})$ est une solution (irréductible) de \eqref{a} sur $R$.

\end{proposition}

\begin{proof}

La preuve donnée dans \cite{M4} s'adapte naturellement au cas d'un sous-magma de $\bC$.

\end{proof}

\begin{proposition}
\label{32}

Soit $G$ un sous-groupe de $\bC$. 
\\i) Une $\lambda$-quiddité de taille supérieure à 5 contenant 0 est réductible.
\\ii) Si $1 \notin G$ alors les $\lambda$-quiddités de taille 4 sont irréductibles.
 
\end{proposition}

\begin{proof}

i) Comme $G$ est un groupe, on a $0 \in G$. Soit $(a_{1},\ldots,a_{n})$ une $\lambda$-quiddité sur $G$ de taille supérieure à 5. On suppose qu'il existe $i$ dans $[\![1;n]\!]$ tel que $a_{i}=0$. On a :

\begin{eqnarray*}
(a_{i+2},\ldots,a_{n},a_{1},\ldots,a_{i},a_{i+1}) &=& (a_{i+2},\ldots,a_{n},a_{1},\ldots,a_{i-1}+a_{i+1}) \\
                                                  &\oplus & (-a_{i+1},0,a_{i+1},0).\\
\end{eqnarray*}

\noindent Or, $(-a_{i+1},0,a_{i+1},0)$ est une $\lambda$-quiddité sur $G$ (voir proposition \ref{31}) et $(a_{i+2},\ldots,a_{n},a_{1},\ldots,a_{i-1}+a_{i+1})$ est de taille $n-2 \geq 3$. Ainsi, $(a_{1},\ldots,a_{n})$ est une $\lambda$-quiddité réductible.
\\
\\ii) Supposons par l'absurde qu'une $\lambda$-quiddité $(c_{1},\ldots,c_{4})$ est réductible. Il existe deux solutions de \eqref{a} $(a_{1},\ldots,a_{l})$ et $(b_{1},\ldots,b_{l'})$ telles que $l, l' \geq 3$ et $(c_{1},\ldots,c_{4}) \sim (a_{1},\ldots,a_{l}) \oplus (b_{1},\ldots,b_{l'})$. On a $6=l+l'$. Ainsi, $l=l'=3$. Or, comme $1 \notin G$ et $-1 \notin G$ (puisque $G$ est un groupe), il n'y a pas de $\lambda$-quiddité de taille 3 sur $G$ (par la proposition \ref{31}). Ceci est absurde.

\end{proof}

Cette proposition montre que la présence de certains éléments dans les $\lambda$-quiddités est un atout important pour l'étude de ces dernières. Il est donc intéressant de chercher des résultats permettant de garantir l'existence de nombres particuliers dans les solutions de \eqref{a}. Dans cet esprit, on dispose du résultat majeur suivant :

\begin{theorem}[Cuntz-Holm, \cite{CH} corollaire 3.3]
\label{33}

Soit $(a_{1},\ldots,a_{n}) \in \bC^{n}$ une $\lambda$-quiddité. $\exists (i,j) \in [\![1;n]\!]^{2}$, $i \neq j$, tels que $\left|a_{i}\right| < 2$ et $\left|a_{j}\right| < 2$.

\end{theorem}

\begin{remark}

{\rm La constante 2 du théorème ci-dessus est optimale. En effet, on peut montrer que, pour tout $\epsilon \in ]0,2]$, il existe une $\lambda$-quiddité sur $\bR$ dont toutes les composantes sont de module supérieur à $2-\epsilon$ (voir \cite{M2} proposition 3.5).
}

\end{remark}

On aura également besoin dans les sections suivantes de connaître une formule où les coefficients de la matrice $M_{n}(a_{1},\ldots,a_{n})$ sont exprimés à l'aide de déterminants. Pour cela, on rappelle la définition classique ci-dessous :
\\
\\On pose $K_{-1}:=0$ et $K_{0}:=1$. Soient $n \in \bM$ et $(a_{1},\ldots,a_{n}) \in \bC^{n}$. On note \[K_{n}(a_{1},\ldots,a_{n}):=
\left|
\begin{array}{cccccc}
a_1&1&&&\\[4pt]
1&a_{2}&1&&\\[4pt]
&\ddots&\ddots&\!\!\ddots&\\[4pt]
&&1&a_{n-1}&\!\!\!\!\!1\\[4pt]
&&&\!\!\!\!\!1&\!\!\!\!a_{n}
\end{array}
\right|.\] $K_{n}(a_{1},\ldots,a_{n})$ est le continuant de $a_{1},\ldots,a_{n}$. On dispose de l'égalité suivante (voir \cite{CO,MO}). Soient $n \in \bM$ et $(a_{1},\ldots,a_{n}) \in \bC^{n}$. On a :
\[M_{n}(a_{1},\ldots,a_{n})=\begin{pmatrix}
    K_{n}(a_{1},\ldots,a_{n}) & -K_{n-1}(a_{2},\ldots,a_{n}) \\
    K_{n-1}(a_{1},\ldots,a_{n-1})  & -K_{n-2}(a_{2},\ldots,a_{n-1}) 
   \end{pmatrix}.\]

Pour calculer le polynôme continuant, on dispose d'un algorithme simple appelé algorithme d'Euler (ou règle d'Euler), apparu lors de l'étude menée par L. Euler des fractions continues. On peut le décrire de la façon suivante (voir par exemple \cite{CO} section 2). $K_{n}(a_{1},\ldots,a_{n})$ est la somme de tous les produits possibles des $a_{1},\ldots,a_{n}$, dans lesquels un nombre quelconque de paires disjointes de termes consécutifs est supprimé. Chacun de ces produits étant multiplié par -1 puissance le nombre de paires supprimées. On commence donc par le produit $a_{1} \times \ldots \times a_{n}$. Ensuite, on soustrait tous les produits de la forme $a_{1} \times \ldots \times a_{i-1} \times a_{i+2} \times \ldots \times a_{n}$. Puis, on additionne tous les produits possibles de $a_{1},\ldots,a_{n}$, dans lesquels deux paires disjointes de termes consécutifs ont été supprimées et on continue ainsi de suite. 

\begin{examples}

{\rm En appliquant cet algorithme, on obtient :}
\begin{itemize}
\item $K_{2}(a_{1},a_{2})=a_{1}a_{2}-1$;
\item $K_{3}(a_{1},a_{2},a_{3})=a_{1}a_{2}a_{3}-a_{3}-a_{1}$;
\item $K_{4}(a_{1},a_{2},a_{3},a_{4})=a_{1}a_{2}a_{3}a_{4}-a_{3}a_{4}-a_{1}a_{4}-a_{1}a_{2}+1$;
\item $K_{5}(a_{1},a_{2},a_{3},a_{4},a_{5})=a_{1}a_{2}a_{3}a_{4}a_{5}-a_{3}a_{4}a_{5}-a_{1}a_{4}a_{5}-a_{1}a_{2}a_{5}-a_{1}a_{2}a_{3}+a_{5}+a_{3}+a_{1}$.
\\
\end{itemize}

\end{examples}

\subsection{Résultats de classification}
\label{class}

On dispose déjà de résultats de classification des $\lambda$-quiddités irréductibles sur certains sous-ensembles de $\bC$, notamment sur $\bN$ et $\bZ$, obtenus à l'aide du théorème \ref{33}.

\begin{theorem}[Cuntz, \cite{C}, Théorème 3.1]
\label{34}

Les $\lambda$-quiddités irréductibles sur $\mathbb{N}$ sont $(1,1,1)$ et $(0,0,0,0)$.

\end{theorem}

\begin{theorem}[Cuntz-Holm, \cite{CH} Théorème 6.2]
\label{35}

Les $\lambda$-quiddités irréductibles sur l'anneau $\mathbb{Z}$ sont : \[\{(1,1,1), (-1,-1,-1), (0,m,0,-m), (m,0,-m,0); m \in \mathbb{Z}-\{\pm 1\} \}.\]

\end{theorem}

Pour ce cas, on dispose d'une description combinatoire élégante des solutions. Avant de donner celle-ci, on définit la notion d'étiquetage admissible introduite dans \cite{CH}. 

\begin{definition}[\cite{CH}, définition 7.1]
\label{36}

Soit $P$ un polygone convexe à $n$ sommets. On considère une triangulation de ce polygone par des diagonales ne se coupant qu'aux sommets et on affecte un entier relatif à chaque triangle. On dit que cette étiquetage est admissible si l'ensemble des triangles étiquetés avec un entier $m \neq \pm 1$ peut être écrit comme une union disjointe d'ensembles contenant chacun deux triangles partageant un côté commun et ayant pour étiquette $a$ et $-a$. À chaque sommet de $P$ on associe un entier $c$ obtenu en sommant les étiquettes des triangles utilisant ce sommet. On parcourt les sommets, à partir de n'importe lequel d'entre eux, dans le sens horaire ou le sens trigonométrique, pour obtenir le $n$-uplet $(c_{1},\ldots,c_{n})$. Ce $n$-uplet est la quiddité de l'étiquetage admissible de $P$.

\end{definition}

\begin{theorem}[Cuntz-Holm, \cite{CH} Théorème 7.3]
\label{37}

Les $\lambda$-quiddités de taille $n$ sur $\mathbb{Z}$ sont des quiddités d'étiquetages admissibles de polygones convexes à $n$ sommets et réciproquement.

\end{theorem}

\begin{examples}

{\rm Voici quelques exemples d'étiquetages admissibles avec leur quiddité :}

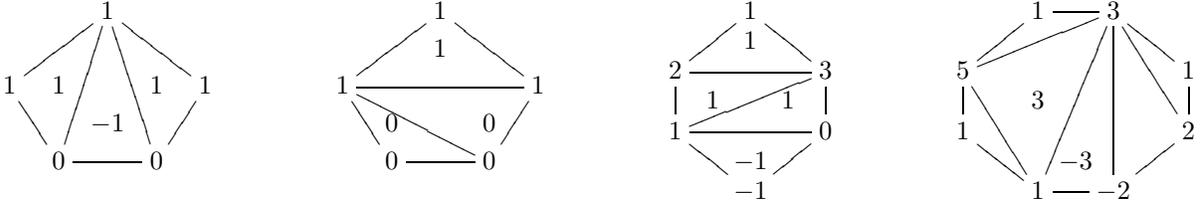
$$
\shorthandoff{; :!?}
\xymatrix @!0 @R=0.50cm @C=0.65cm
{
&& 1 \ar@{-}[rrdd]\ar@{-}[lldd]\ar@{-}[ldddd]\ar@{-}[rdddd]&
\\
\\
1 \ar@{-}[rdd]&1&&1& 1 \ar@{-}[ldd]
\\
&&-1
\\
&0\ar@{-}[rr]&& 0
}
\qquad
\qquad
\xymatrix @!0 @R=0.50cm @C=0.65cm
{
&& 1 \ar@{-}[rrdd]\ar@{-}[lldd]&
\\
&&1
\\
1 \ar@{-}[rdd]\ar@{-}[rrrdd]\ar@{-}[rrrr]&&&& 1 \ar@{-}[ldd]
\\
&0&&0
\\
&0\ar@{-}[rr]&& 0
}
\qquad
\qquad
\xymatrix @!0 @R=0.40cm @C=0.5cm
{
&&1 \ar@{-}[rrdd]\ar@{-}[lldd]&
\\
&&1&
\\
2\ar@{-}[dd]\ar@{-}[rrrr]&&&& 3 \ar@{-}[dd]
\\
&1 && 1
\\
1\ar@{-}[rrrr]\ar@{-}[rrrruu]&&&& 0
\\
&&-1&
\\
&&-1 \ar@{-}[lluu]\ar@{-}[rruu]
}
\qquad
\qquad
\xymatrix @!0 @R=0.40cm @C=0.5cm
{
&&1\ar@{-}[lldd] \ar@{-}[rr]&&3\ar@{-}[rrdd]\ar@{-}[dddddd]\ar@{-}[lldddddd]\ar@{-}[rrdddd]&
\\
&&&
\\
5\ar@{-}[dd]\ar@{-}[rrrruu]&&&&&& 1 \ar@{-}[dd]
\\
&&3
\\
1&&&&&& 2
\\
&&&-3
\\
&&1 \ar@{-}[rr]\ar@{-}[lluuuu]\ar@{-}[lluu] &&-2 \ar@{-}[rruu]
}
$$

\noindent {\rm Dans le dernier exemple, tous les triangles sont étiquetés avec 1, à l'exception des deux triangles déjà étiquetés sur la figure. Remarquons que les deux pentagones ont un étiquetage admissible différent aboutissant à la même quiddité.}

\end{examples}

\begin{remark}
{\rm La version initiale du théorème telle qu'elle apparaît dans \cite{CH} concerne uniquement les solutions de $M_{n}(a_{1},\ldots,a_{n})=-Id$ et introduit pour cela une notion de signe de l'étiquetage.
}
\end{remark}

\indent On dispose également du résultat de classification énoncé ci-dessous. Celui-ci concerne le cas des sous-anneaux $\bZ[\alpha]$ avec $\alpha$ un nombre complexe transcendant, c'est-à-dire un nombre complexe qui n'est racine d'aucun polynôme non nul à coefficients dans $\bZ$.

\begin{theorem}[\cite{M2}, Théorème 2.7]
\label{38}

Soit $\alpha$ un nombre complexe transcendant. Les $\lambda$-quiddités irréductibles sur l'anneau $\bZ[\alpha]$ sont : \[\{(1,1,1), (-1,-1,-1), (0,P(\alpha),0,-P(\alpha)), (P(\alpha),0,-P(\alpha),0); P \in \bZ[X]-\{\pm 1\} \}.\]

\end{theorem}

\subsection{Premiers éléments sur les sous-groupes monogènes}
\label{mono}

Nous allons revenir à notre problème initial en donnant quelques résultats de classification des $\lambda$-quiddités irréductibles sur $<w>$ pour certains nombres complexes $w$.

\begin{proposition}
\label{39}

Soit $w$ un nombre complexe vérifiant $\left|w\right| \geq 2$. L'ensemble des $\lambda$-quiddités irréductibles sur $<w>$ est : \[\{(0,kw,0,-kw), (kw,0,-kw,0); k \in \bZ\}.\]

\end{proposition}

\begin{proof}

Soit $(k_{1}w,\ldots,k_{n}w)$ une $\lambda$-quiddité sur $<w>$. D'après le théorème \ref{33}, il existe $i$ dans $[\![1;n]\!]$ tel que $\left|k_{i}w\right| < 2$. Or, $\left|k_{i}w\right|=\left|k_{i}\right|\left|w\right| \geq 2\left|k_{i}\right|$. Donc, $k_{i}=0$. Par la proposition \ref{32} i), les $\lambda$-quiddités sur $<w>$ de taille supérieure à 5 sont réductibles. 
\\
\\Par la proposition \ref{31}, il n'y a pas de $\lambda$-quiddité sur $<w>$ de taille 3 (car $1 \notin <w>$). Donc, par la proposition \ref{32} ii), les $\lambda$-quiddités irréductibles sur $<w>$ sont exactement les $\lambda$-quiddités de taille 4. Or, $ab=2$ n'a pas de solution dans $<w>$. Ainsi, les solutions de \eqref{a} de taille 4 sont de la forme $(0,kw,0,-kw)$ ou $(kw,0,-kw,0)$ (avec $k \in \bZ$). 

\end{proof}

\begin{corollary}
\label{310}

Soit $a \in \bZ$.
\\
\\i) Si $a=0$. $(0,0,0,0)$ est la seule $\lambda$-quiddité irréductible sur $<a>$.
\\
\\ii) Si $a=\pm 1$. L'ensemble des $\lambda$-quiddités irréductibles sur $<a>$ est : \[\{(1,1,1), (-1,-1,-1), (0,m,0,-m), (m,0,-m,0); m \in \bZ-\{\pm 1\} \}.\]

\noindent iii) Si $\left|a\right| \geq 2$. L'ensemble des $\lambda$-quiddités irréductibles sur $<a>$ est : \[\{(0,ka,0,-ka), (ka,0,-ka,0); k \in \bZ\}.\]

\end{corollary}

\begin{proof}

i) $<a>=\{0\}$. Par la proposition \ref{32} i), les $\lambda$-quiddités sur $<a>$ de taille supérieure à 5 sont réductibles. Par la proposition \ref{31}, les seules $\lambda$-quiddités sur $<a>$ de taille inférieure à 4 sont $(0,0)$ et $(0,0,0,0)$. Par la proposition \ref{32} ii), $(0,0,0,0)$ est la seule $\lambda$-quiddité irréductible sur $<a>$.
\\
\\ii) $<a>=\bZ$. Le résultat est déjà connu (voir théorème \ref{35}).
\\
\\iii) C'est une conséquence de la proposition précédente.

\end{proof}

\noindent On va maintenant s'intéresser aux cas où $w$ est un nombre complexe transcendant.

\begin{proposition}
\label{311}

Soit $\alpha$ un nombre complexe transcendant. L'ensemble des $\lambda$-quiddités irréductibles sur $<\alpha>$ est : \[\{(0,k\alpha,0,-k\alpha), (k\alpha,0,-k\alpha,0); k \in \bZ\}.\]

\end{proposition}

\begin{proof}

Soient $n \in \bN$ $n \geq 2$ et $(k_{1}\alpha,\ldots,k_{n}\alpha)$ une $\lambda$-quiddité sur $<\alpha>$. il existe $\epsilon$ dans $\{-1, 1\}$ tel que :
\[M_{n}(k_{1}\alpha,\ldots,k_{n}\alpha)=\begin{pmatrix}
    K_{n}(k_{1}\alpha,\ldots,k_{n}\alpha) & -K_{n-1}(k_{2}\alpha,\ldots,k_{n}\alpha) \\
    K_{n-1}(k_{1}\alpha,\ldots,k_{n-1}\alpha)  & -K_{n-2}(k_{2}\alpha,\ldots,k_{n-1}\alpha) 
   \end{pmatrix}=\epsilon Id.\]
	
\noindent Supposons par l'absurde $n$ impair. Par l'algorithme d'Euler, $K_{n-1}(k_{1}\alpha,\ldots,k_{n-1}\alpha)$ est un polynôme en $\alpha$ à coefficients dans $\bZ$ (dépendant des $k_{i}$) dont le coefficient constant est $\pm 1$ (car $n-1$ est pair). Or, $P(\alpha)=K_{n-1}(k_{1}\alpha,\ldots,k_{n-1}\alpha)=0$. Ainsi, comme $\alpha$ est transcendant, $P$ est un polynôme nul. Ceci est absurde car le coefficient constant de $P$ est $\pm 1$.
\\
\\On en déduit que $n$ est pair.
\\
\\Par l'algorithme d'Euler, $K_{n-1}(k_{1}\alpha,\ldots,k_{n-1}\alpha)$ est un polynôme en $\alpha$ à coefficients dans $\bZ$ (dépendant des $k_{i}$) dont le coefficient dominant est $\prod_{i=1}^{n-1} k_{i}$. Or, $P(\alpha)=K_{n-1}(k_{1}\alpha,\ldots,k_{n-1}\alpha)=0$. Ainsi, comme $\alpha$ est transcendant, $P$ est un polynôme nul. Donc, $\prod_{i=1}^{n-1} k_{i}=0$, c'est-à-dire qu'il existe $i$ dans $[\![1;n-1]\!]$ tel que $k_{i}=0$. Par la proposition \ref{32} i), les $\lambda$-quiddités sur $<\alpha>$ de taille supérieure à 5 sont réductibles.
\\
\\Comme $\alpha$ est transcendant, $1 \notin <\alpha>$ et $ab=2$ n'a pas de solution dans $<\alpha>$. Donc, par la proposition \ref{32} ii), les $\lambda$-quiddités irréductibles sur $<\alpha>$ sont de la forme $(0,k\alpha,0,-k\alpha)$ ou $(k\alpha,0,-k\alpha,0)$ ($k \in \bZ$).

\end{proof}

Cette proposition couvre de très nombreux cas. En effet, l'ensemble des nombres complexes algébriques est dénombrable (voir par exemple \cite{Ca} et \cite {G} corollaire III.57) alors que $\bC$ est indénombrable (Théorème de Cantor, voir \cite{G} corollaire III.54). De plus, il existe de nombreux résultats permettant de considérer des nombres transcendants précis. Par exemple :
\begin{itemize}
\item Nombres de Liouville (1844) : Si $(a_{n})_{n \in \bN}$ est une suite d'entiers compris entre 0 et 9 avec une infinité de terme non nuls alors la série numérique de terme générale $\frac{a_{n}}{10^{n!}}$ est convergente et $\sum_{n=0}^{+\infty} \frac{a_{n}}{10^{n!}}$ est transcendant (voir \cite{L} et \cite{G} exercice III.11).
\\
\item Théorème de Hermite (1873) : $e$ est transcendant (voir \cite{H} et \cite{G} Théorème III.59).
\\
\item Théorème de Lindemann (1882) : $\pi$ est transcendant (voir \cite{G} Théorème III.60).
\\
\item Théorème de Gelfond-Schneider (1934) : Si $a$ est un nombre algébrique différent de 0 et 1 et si $b$ est un nombre algébrique irrationnel alors $a^{b}$ est transcendant, où $a^{b}$ est à prendre au sens ${\rm exp}(b{\rm log}(a))$ avec ${\rm log}(a)$ n'importe quelle détermination du logarithme complexe de $a$ (voir \cite{Ge}).
\\
\end{itemize}

Nous allons nous intéresser dans les sections suivantes aux cas de certains nombres complexes algébriques en démontrant les deux théorèmes présentés dans la section \ref{RP}.

\subsection{Solutions dont toutes les composantes sont identiques}

Parmi les $\lambda$-quiddités irréductibles listées dans le théorème \ref{25}, on constate la présence de solutions dont toutes les composantes sont identiques. On va chercher ici, pour un entier $n$ fixé, quels sont les nombres complexes $z$ pour lesquels le $n$-uplet $(z,\ldots,z)$ est une $\lambda$-quiddité irréductible sur tout sous-groupe de $\mathbb{C}$ contenant $z$.

\begin{proposition}

Soient $n \geq 2$ et $z \in \mathbb{C}$. 
\\
\\i) Le $n$-uplet $(z,\ldots,z)$ est une $\lambda$-quiddité sur $\mathbb{C}$ si et seulement s'il existe $1 \leq k \leq n-1$ tel que $z=2cos\left(\frac{k\pi}{n}\right)$.
\\
\\ii) Soient $1 \leq k \leq n-1$ et $G$ un sous-groupe de $\mathbb{C}$ contenant $z:=2cos\left(\frac{k\pi}{n}\right)$. $(z,\ldots,z) \in G^{n}$ est irréductible sur $G$ si et seulement si $n=4$ et $k=2$ ou $n \geq 3$ et $k$ et $n$ premiers entre eux.

\end{proposition}

\begin{proof}

i) Supposons que le $n$-uplet $(z,\ldots,z)$ est une $\lambda$-quiddité sur $\mathbb{C}$. On a $K_{n-1}(z,\ldots,z)=0$. Or, $K_{n-1}(X,\ldots,X)=U_{n-1}\left(\frac{X}{2}\right)$, avec $U_{n}$ le polynôme de Tchebychev de seconde espèce de degré $n$ (voir \cite{CO} section 5). Donc, $\frac{z}{2}$ est une racine de $U_{n-1}$. Or, $U_{n-1}$ est un polynôme de degré $n-1$ qui possèdent $n-1$ racines réelles distinctes : les  ${\rm cos}\left(\frac{k\pi}{n}\right)$ avec $1 \leq k \leq n-1$. En effet, $U_{n-1}$ est un polynôme de degré $n-1$ donc il a au plus $n-1$ racines complexes et, en utilisant la relation $U_{n}({\rm cos}(\theta))=\frac{{\rm sin}((n+1)\theta)}{{\rm sin}(\theta)}$ pour tout $\theta \not\equiv 0 [\pi]$, on a $U_{n-1}\left({\rm cos}\left(\frac{k\pi}{n}\right)\right)=\frac{{\rm sin}(k\pi)}{{\rm sin}(\frac{k\pi}{n})}=0$. Comme ${\rm cos}\left(\frac{k\pi}{n}\right) \neq {\rm cos}\left(\frac{l\pi}{n}\right)$ lorsque $1 \leq k <l \leq n-1$, on obtient le résultat souhaité. Par conséquent, $z$ a la forme annoncée.
\\
\\Soit $z:=2{\rm cos}\left(\frac{k\pi}{n}\right)$ avec $1 \leq k \leq n-1$. Par ce qui précède, $K_{n-1}(z,\ldots,z)=0$. De plus, on a :
\[K_{n-2}(z,\ldots,z)=U_{n-2}\left({\rm cos}\left(\frac{k\pi}{n}\right)\right)=\frac{{\rm sin}\left(k\pi-\frac{k\pi}{n}\right)}{{\rm sin}\left(\frac{k\pi}{n}\right)}=(-1)^{k+1}\frac{{\rm sin}\left(\frac{k\pi}{n}\right)}{{\rm sin}\left(\frac{k\pi}{n}\right)}=(-1)^{k+1};\]
\[K_{n}(z,\ldots,z)=U_{n}\left({\rm cos}\left(\frac{k\pi}{n}\right)\right)=\frac{{\rm sin}\left(k\pi+\frac{k\pi}{n}\right)}{{\rm sin}\left(\frac{k\pi}{n}\right)}=(-1)^{k}\frac{{\rm sin}\left(\frac{k\pi}{n}\right)}{{\rm sin}\left(\frac{k\pi}{n}\right)}=(-1)^{k}.\]
\noindent Par conséquent, $M_{n}(z,\ldots,z)=(-1)^{k}Id$.
\\
\\ii) Soit $z:=2{\rm cos}\left(\frac{k\pi}{n}\right)$ avec $1 \leq k \leq n-1$. Soit $G$ un sous-groupe de $\mathbb{C}$ contenant $z$. 
\\
\\Supposons que $k$ et $n$ ne sont pas premiers entre eux. Soit $u$ le pgcd de $k$ et $n$. Il existe deux entiers premiers entre eux $k',n'$ tels que $n=n'u$ et $k=k'u$. On a donc $z=2{\rm cos}\left(\frac{k'\pi}{n'}\right)$. De plus, comme $0<k<n$, on a $n' \neq 1$. On distingue deux cas :
\begin{itemize}
\item Supposons que $n'=2$. Comme $0<k<n$, on a nécessairement $k=u$ et $z=0$. Si $u=2$ alors $n=4$ et la solution $(z,\ldots,z) \in G^{n}$ est irréductible. Si $u>2$ alors $n \geq 6$ et on peut réduire la solution avec la $\lambda$-quiddité $(0,0,0,0)$ (proposition \ref{32} i)).
\item Supposons que $n'>2$. Par i), le $n'$-uplet $(z,\ldots,z)$ est une $\lambda$-quiddité sur le groupe $G$. Comme $3 \leq n' \leq \frac{n}{2} \leq n-1$, on peut utiliser cette $\lambda$-quiddité de taille $n'$ pour réduire notre $\lambda$-quiddité de taille $n$.
\\
\end{itemize}

\noindent Supposons maintenant que $k$ et $n$ sont premiers entre eux. Si $n=2$ alors $k=1$, $z=0$ et la solution est réductible. Si $n=3$ alors $k \in \{1,2\}$, $z=\pm 1$ et la solution est irréductible. On suppose maintenant $n \geq 4$ et on suppose par l'absurde que $(z,\ldots,z) \in G^{n}$ est réductible. Il existe une solution sur le groupe $G$ de la forme $(a,z,\ldots,z,b)$ de taille $3 \leq l \leq n-1$. On a $K_{l-2}(z,\ldots,z)= \pm 1$. Donc, 
\[U_{l-2}\left({\rm cos}\left(\frac{k\pi}{n}\right)\right)=\frac{{\rm sin}\left(\frac{(l-1)k\pi}{n}\right)}{{\rm sin}\left(\frac{k\pi}{n}\right)}=\pm 1.\]
\noindent Ainsi, ${\rm sin}\left(\frac{(l-1)k\pi}{n}\right)=\pm {\rm sin}\left(\frac{k\pi}{n}\right)$, ce qui implique $\frac{lk\pi}{n} \equiv 0 [\pi]$ ou $\frac{(l-2)k\pi}{n} \equiv 0 [\pi]$. On distingue les deux cas :
\begin{itemize}
\item Supposons que $\frac{lk\pi}{n} \equiv 0 [\pi]$. On a $lk \equiv 0 [n]$, c'est-à-dire $n$ divise $lk$. Comme $k$ et $n$ sont premiers entre eux, $n$ divise $l$, par le lemme de Gauss. Comme $l <n$, ceci est absurde.
\item Supposons que $\frac{(l-2)k\pi}{n} \equiv 0 [\pi]$. On a $(l-2)k \equiv 0 [n]$, c'est-à-dire $n$ divise $(l-2)k$. Comme $k$ et $n$ sont premiers entre eux, $n$ divise $l-2$, par le lemme de Gauss. Comme $l <n$, ceci est absurde.
\end{itemize}
\noindent Donc, $(z,\ldots,z) \in G^{n}$ est irréductible.

\end{proof}

\begin{examples}
{\rm \begin{itemize}
\item Soit $\tau:=\frac{1+\sqrt{5}}{2}$ le nombre d'or. On a $\tau=2{\rm cos}\left(\frac{\pi}{5}\right)$. Par conséquent, $(\tau,\tau,\tau,\tau,\tau)$ est une $\lambda$-quiddité irréductible sur $<\tau>$.
\item $\left(2{\rm cos}\left(\frac{5\pi}{7}\right),2{\rm cos}\left(\frac{5\pi}{7}\right),2{\rm cos}\left(\frac{5\pi}{7}\right),2{\rm cos}\left(\frac{5\pi}{7}\right),2{\rm cos}\left(\frac{5\pi}{7}\right),2{\rm cos}\left(\frac{5\pi}{7}\right),2{\rm cos}\left(\frac{5\pi}{7}\right)\right)$ est une $\lambda$-quiddité irréductible sur $\mathbb{Z}\left[2{\rm cos}\left(\frac{5\pi}{7}\right)\right]$.
\end{itemize}
}
\end{examples}

\begin{remarks}
{\rm i) Pour $n=4$, on obient les $\lambda$-quiddités irréductibles $(0,0,0,0)$, $\pm (\sqrt{2},\sqrt{2},\sqrt{2},\sqrt{2})$. Pour $n=3$, on obtient les solutions irréductibles $\pm (1,1,1)$. Pour $n=6$, on obtient les $\lambda$-quiddités irréductibles $\pm (\sqrt{3},\sqrt{3},\sqrt{3},\sqrt{3},\sqrt{3},\sqrt{3})$. On retrouve donc les solutions dont toutes les composantes sont identiques présentes dans le théorème \ref{25}. Cependant, on n'utilisera pas les résultats de cette sous-partie dans la preuve de ce dernier.
\\
\\ii) Lorsque l'on travaille sur les anneaux $\mathbb{Z}/N\mathbb{Z}$, on dispose de très nombreux résultats sur les $\lambda$-quiddités dont toutes les composantes sont identiques (voir notamment \cite{M4,M5}).
}
\end{remarks}

\section{Démonstration du théorème \ref{25}}
\label{preuve25}

L'objectif de cette section est de démontrer le théorème \ref{25}, c'est-à-dire de considérer les cas $w=\sqrt{k}$ avec $k \in \bN$. Le point i) correspond exactement au point i) du corollaire \ref{310}. De même, le point ii) est scrupuleusement identique au théorème \ref{35}. Quant au point v), il découle directement de la proposition \ref{39}. Il nous reste donc à considérer les cas des sous-groupes $<\sqrt{k}>$ avec $k=2$ et $k=3$.
\\
\\Dans ce qui suit, si $(a_{1},\ldots a_{n})$ est une solution de \eqref{a} on considérera $a_{1}$ et $a_{n}$ comme des éléments consécutifs (ceci à cause de l'invariance par permutations circulaires des solutions de \eqref{a}).

\subsection{Le cas k=2}

Soient $m$ un entier strictement positif et $(a_{1},\ldots,a_{m})$ une solution de \eqref{a}, avec, pour tout $j$ dans $[\![1;m]\!]$, $a_{j}=k_{j}\sqrt{2}$ et $k_{j} \in \mathbb{Z}$. Il existe $\epsilon \in \{\pm 1 \}$ tel que : \[\epsilon Id=M_{m}(a_{1},\ldots,a_{m})=\begin{pmatrix}
    K_{m}(a_{1},\ldots,a_{m}) & -K_{m-1}(a_{2},\ldots,a_{m}) \\
    K_{m-1}(a_{1},\ldots,a_{m-1})  & -K_{m-2}(a_{2},\ldots,a_{m-1}) 
   \end{pmatrix}.\]
\noindent Supposons $m$ impair. On a $K_{m}(a_{1},\ldots,a_{m})=\epsilon \in \mathbb{Z}$.
\\
\\Par l'algorithme d'Euler, $K_{m}(a_{1},\ldots,a_{m})$ est une somme dont chacun des termes est un produit d'un nombre impair de $a_{j}$ (et éventuellement de -1). On regroupe les $\sqrt{2}$ présents dans chaque terme. Chacun des termes de la somme est alors le produit de $\sqrt{2}$ à une puissance impaire multiplié par des entiers. En factorisant par $\sqrt{2}$, on peut donc écrire $K_{m}(a_{1},\ldots,a_{m})$ comme le produit de $\sqrt{2}$ par un entier.
\\
\\Ainsi, $K_{m}(a_{1},\ldots,a_{m}) \in \mathbb{Z} \cap \sqrt{2}\mathbb{Z}$, c'est-à-dire $K_{m}(a_{1},\ldots,a_{m})=0$. Ceci est absurde puisque $\epsilon \in \{\pm 1 \}$.
\\
\\Donc, il n'y a pas de $\lambda$-quiddité de taille impaire sur $<\sqrt{2}>$.
\\
\\Supposons maintenant $m$ pair supérieur à 4 avec $m=2n$. Notre objectif est de montrer que si tous les $k_{i}$ sont non nuls alors deux $k_{i}$ consécutifs sont égaux à 1 ou à -1.
\\
\\Par l'algorithme d'Euler, $K_{2n}(k_{1}\sqrt{2},k_{2}\sqrt{2},\ldots,k_{2n-1}\sqrt{2},k_{2n}\sqrt{2})$ est une somme de produits des $k_{j}\sqrt{2}$ (et éventuellement de -1). Dans chaque terme d'une de ces sommes, pour chaque $j$ impair $k_{j}$ est suivi d'un $k_{l}$ avec $l$ pair. On fait alors la manipulation suivante :\[k_{j}\sqrt{2}k_{l}\sqrt{2}=2k_{j}k_{l}=k_{j}(2k_{l}).\]
\noindent Avec cette manipulation, on a \[K_{2n}(k_{1}\sqrt{2},k_{2}\sqrt{2},\ldots,k_{2n-1}\sqrt{2},k_{2n}\sqrt{2})=K_{2n}(k_{1},2k_{2},\ldots,k_{2n-1},2k_{2n}).\] 

\noindent On procède de façon analogue pour $K_{2n-2}(k_{2}\sqrt{2},\ldots,k_{2n-1}\sqrt{2})$ (puisque pour pour chaque $j$ impair $k_{j}$ est précédé d'un $k_{l}$ avec $l$ pair). 
\\
\\$K_{2n-1}(k_{1}\sqrt{2},k_{2}\sqrt{2},\ldots,k_{2n-1}\sqrt{2})$ est une somme de produits des $k_{j}\sqrt{2}$ (et éventuellement de -1). Pour chaque terme de la somme, on a deux choix :
\begin{itemize}
\item le terme est de la forme $k_{j}\sqrt{2}$ avec $j$ impair;
\item le terme est un produit de $2u+1$ éléments de la forme $k_{j}\sqrt{2}$ avec $u+1$ indices $j$ impair. Pour chaque $m$ pair $k_{m}$ est précédé d'un $k_{l}$ avec $l$ impair. On fait alors la manipulation suivante :\[k_{l}\sqrt{2}k_{m}\sqrt{2}=2k_{l}k_{m}=k_{l}(2k_{m}).\]
\end{itemize}
\noindent Cela donne $0=K_{2n-1}(k_{1}\sqrt{2},k_{2}\sqrt{2},\ldots,k_{2n-1}\sqrt{2})=\sqrt{2}K_{2n-1}(k_{1},2k_{2},\ldots,k_{2n-1})$. En particulier, on a $0=K_{2n-1}(k_{1},2k_{2},\ldots,k_{2n-1})$. On procède de façon analogue pour $K_{2n-1}(k_{2}\sqrt{2},\ldots,k_{2n}\sqrt{2})$.
\\
\\Donc, $(k_{1},2k_{2},\ldots,k_{2n-1},2k_{2n})$ est une $\lambda$-quiddité sur $\mathbb{Z}$.
\\
\\On suppose que pour tout $i \in [\![1;2n]\!]$, $k_{i} \neq 0$. Ainsi, par le théorème \ref{33},  il existe $i \in [\![1;2n]\!]$ tel que $k_{i}=\pm 1$ et $i$ impair. 
\\
\\On montre, par un simple calcul, que $M_{3}(a,1,b)=M_{2}(a-1,b-1)$ et $M_{3}(a,-1,b)=-M_{2}(a+1,b+1)$. Si on "réduit" de cette façon tous les $\pm 1$ présents dans $(k_{1},2k_{2},\ldots,k_{2n-1},2k_{2n})$ on obtient une nouvelle solution. Cette nouvelle $\lambda$-quiddité ne peut contenir que des éléments appartenant à la liste suivante :
\begin{itemize}
\item  $k_{i}$ avec $i$ impair et $\left|k_{i}\right| \geq 2$;
\item $2k_{i}$ avec $i$ pair;
\item $2k_{i}-1$ avec $i$ pair;
\item $2k_{i}+1$ avec $i$ pair;
\item $2k_{i}-2$ avec $i$ pair;
\item $2k_{i}+2$ avec $i$ pair.
\\
\end{itemize}

\noindent Par le théorème \ref{33}, un de ces éléments est nécessairement égal à 0, 1 ou -1. Cela ne peut pas être un des éléments des deux premières catégories. S'il existe $i \in [\![1;2n]\!]$ tel que $2k_{i}-1 \in \{0, 1, -1\}$. On a nécessairement $2k_{i}-1=1$, c'est-à-dire $k_{i}=1$. De plus, pour arriver à $2k_{i}-1$ lors de la réduction on a impérativement utilisé un 1 adjacent à $2k_{i}$. Ainsi, on a deux $k_{i}$ consécutifs égaux à 1. On procède de façon analogue s'il existe $i \in [\![1;2n]\!]$ tel que $2k_{i}+1 \in \{0, 1, -1\}$. Supposons maintenant qu'il existe $i \in [\![1;2n]\!]$ tel que $2k_{i}-2 \in \{0, 1, -1\}$. On a nécessairement $2k_{i}-2=0$, c'est-à-dire $k_{i}=1$. De plus, pour arriver à $2k_{i}-2$ lors de la réduction on a obligatoirement utilisé un 1 adjacent à gauche et à droite de $2k_{i}$. Ainsi, on a deux $k_{i}$ consécutifs égaux à 1. On procède de façon analogue s'il existe $i \in [\![1;2n]\!]$ tel que $2k_{i}+2 \in \{0, 1, -1\}$.
\\
\\Ainsi, si $(a_{1},\ldots,a_{n})$ est une $\lambda$-quiddité sur $<\sqrt{2}>$ alors $n$ est pair et une des deux conditions suivantes est obligatoirement vérifiée :
\begin{itemize}
\item un des $a_{i}$ est nul;
\item deux $a_{i}$ consécutifs sont égaux et valent $\sqrt{2}$ ou $-\sqrt{2}$.
\\
\end{itemize}

\noindent Par la proposition \ref{31}, $(\sqrt{2},\sqrt{2},\sqrt{2},\sqrt{2})$ et $(-\sqrt{2},-\sqrt{2},-\sqrt{2},-\sqrt{2})$ sont des $\lambda$-quiddités sur $<\sqrt{2}>$. Aussi, on peut les utiliser pour réduire toutes les solutions de \eqref{a} de taille supérieure à 5 dont deux éléments consécutifs sont égaux et valent $\sqrt{2}$ ou $-\sqrt{2}$. 
\\
\\Donc, les $\lambda$-quiddités sur $<\sqrt{2}>$ de taille supérieure à 5 sont réductibles (par la proposition \ref{32} i) et la discussion précédente). Par les propositions \ref{31} et \ref{32} ii), les $\lambda$-quiddités irréductibles sur $<\sqrt{2}>$ sont celles données dans l'énoncé.

\qed

\subsection{Le cas k=3}

Soient $m$ un entier strictement positif et $(a_{1},\ldots,a_{m})$ une solution de \eqref{a}, avec, pour tout $j$ dans $[\![1;m]\!]$, $a_{j}=k_{j}\sqrt{3}$ et $k_{j} \in \mathbb{Z}$. En procédant de la même manière que dans le cas précédent, on obtient $m$ pair, c'est-à-dire $m=2n$, et $(k_{1},3k_{2},\ldots,k_{2n-1},3k_{2n})$ solution de \eqref{a} sur $\mathbb{Z}$. 
\\
\\Par la proposition \ref{31}, les seules $\lambda$-quiddités sur $<\sqrt{3}>$ de taille inférieure à 4 sont celles données dans l'énoncé. Par la proposition \ref{32} ii), celles-ci sont irréductibles. On suppose maintenant $m$ pair supérieur à 6. Notre objectif est de montrer que si tous les $k_{i}$ sont non nuls alors quatre $k_{i}$ consécutifs sont égaux à 1 ou à -1.
\\
\\On suppose que pour tout $i \in [\![1;2n]\!]$, $k_{i} \neq 0$. Ainsi, par le théorème \ref{33},  il existe $i \in [\![1;2n]\!]$ tel que $k_{i}=\pm 1$ et $i$ impair. 
\\
\\On montre, par un simple calcul, que $M_{3}(a,1,b)=M_{2}(a-1,b-1)$ et $M_{3}(a,-1,b)=-M_{2}(a+1,b+1)$. Si on "réduit" de cette façon tous les $\pm 1$ présents dans $(k_{1},3k_{2},\ldots,k_{2n-1},3k_{2n})$ on obtient une nouvelle solution. Cette nouvelle $\lambda$-quiddité ne peut contenir que des éléments appartenant à la liste suivante :
\begin{itemize}
\item  $k_{i}$ avec $i$ impair et $\left|k_{i}\right| \geq 2$;
\item $3k_{i}$ avec $i$ pair;
\item $3k_{i}-1$ avec $i$ pair;
\item $3k_{i}+1$ avec $i$ pair;
\item $3k_{i}-2$ avec $i$ pair;
\item $3k_{i}+2$ avec $i$ pair.
\\
\end{itemize}

\noindent Par le théorème \ref{33}, un de ces éléments est nécessairement égal à 0, 1 ou -1. Comme pour tout $i$ dans $[\![1;2n]\!]$ $k_{i} \neq 0$, cela ne peut pas être un des éléments des quatre premières catégories. S'il existe $i \in [\![1;2n]\!]$ tel que $3k_{i}-2 \in \{0, 1, -1\}$. On a nécessairement $3k_{i}-2=1$, c'est-à-dire $k_{i}=1$. De plus, pour arriver à $3k_{i}-2$ lors de la réduction on a impérativement utilisé un 1 adjacent à droite et à gauche de $3k_{i}$. Ainsi, on a trois $k_{i}$ consécutifs égaux à 1. On procède de façon analogue s'il existe $i \in [\![1;2n]\!]$ tel que $3k_{i}+2 \in \{0, 1, -1\}$. 
\\
\\On "réduit" à nouveau tous les $\pm 1$ présents dans la $\lambda$-quiddité obtenue précédemment. On obtient une nouvelle solution qui ne peut contenir que des éléments appartenant à la liste suivante :
\begin{itemize}
\item  $k_{i}$ avec $i$ impair et $\left|k_{i}\right| \geq 2$;
\item $3k_{i}$ avec $i$ pair;
\item $3k_{i}-1$ avec $i$ pair;
\item $3k_{i}+1$ avec $i$ pair;
\item $3k_{i}-2$ avec $i$ pair;
\item $3k_{i}+2$ avec $i$ pair;
\item $3k_{i}-3$ avec $i$ pair;
\item $3k_{i}+3$ avec $i$ pair;
\item $3k_{i}-4$ avec $i$ pair;
\item $3k_{i}+4$ avec $i$ pair.
\\
\end{itemize}

\noindent Par le théorème \ref{33}, un de ces éléments est nécessairement égal à 0, 1 ou -1. Comme pour tout $i$ dans $[\![1;2n]\!]$ $k_{i} \neq 0$, cela ne peut pas être un des éléments des quatre premières catégories. 
\\
\\Supposons maintenant que $3k_{i}-2 \in \{0, 1, -1\}$ pour un certain $i$ pair. On a nécessairement $3k_{i}-2=1$, c'est-à-dire $k_{i}=1$. De plus, pour arriver à $3k_{i}-2$ lors de la réduction on a obligatoirement utilisé un 1 adjacent à gauche ou à droite d'un $3k_{i}$ ou d'un $3k_{i}-1$. Par ce qui précède, ce 1 provenait d'un triplet de $k_{i}$ consécutifs égaux à 1. On a donc au moins quatre 1 consécutifs.
\\
\\Supposons maintenant que $3k_{i}+2 \in \{0, 1, -1\}$ pour un certain $i$ pair. On a nécessairement $3k_{i}+2=-1$, c'est-à-dire $k_{i}=-1$. De plus, pour arriver à $3k_{i}+2$ lors de la réduction on a obligatoirement utilisé un -1 adjacent à gauche ou à droite d'un $3k_{i}$ ou d'un $3k_{i}+1$ . Par ce qui précède, ce -1 provenait d'un triplet de $k_{i}$ consécutifs égaux à -1. On a donc au moins quatre -1 consécutifs.
\\
\\Supposons maintenant que $3k_{i}+3 \in \{0, 1, -1\}$ pour un certain $i$ pair. On a nécessairement $3k_{i}+3=0$, c'est-à-dire $k_{i}=-1$. De plus, pour arriver à $3k_{i}+3$ lors de la réduction on a obligatoirement utilisé un -1 adjacent à gauche ou à droite d'un $3k_{i}+1$ ou d'un $3k_{i}+2$. Par ce qui précède, ce -1 provenait d'un triplet de $k_{i}$ consécutifs égaux à -1. On a donc au moins quatre -1 consécutifs. On procède de façon analogue si $3k_{i}-3 \in \{0, 1, -1\}$ pour un certain $i$ pair.
\\
\\Supposons maintenant que $3k_{i}-4 \in \{0, 1, -1\}$ pour un certain $i$ pair. On a nécessairement $3k_{i}-4=-1$, c'est-à-dire $k_{i}=1$. De plus, pour arriver à $3k_{i}-4$ lors de la réduction on a obligatoirement utilisé un 1 adjacent à gauche et à droite d'un $3k_{i}-2$. Par ce qui précède, chacun de ces 1 provenait d'un triplet de $k_{i}$ consécutifs égaux à 1. On a donc au moins quatre 1 consécutifs. On procède de façon analogue si $3k_{i}+4 \in \{0, 1, -1\}$ pour un certain $i$ pair.
\\
\\Ainsi, si $(a_{1},\ldots,a_{n})$ est une $\lambda$-quiddité sur $<\sqrt{3}>$ alors $n$ est pair et une des deux conditions suivantes est obligatoirement vérifiée :
\begin{itemize}
\item un des $a_{i}$ est nul;
\item quatre $a_{i}$ consécutifs sont égaux et valent $\sqrt{3}$ ou $-\sqrt{3}$.
\\
\end{itemize}

\noindent Par un simple calcul, on montre que $(\sqrt{3},\sqrt{3},\sqrt{3},\sqrt{3},\sqrt{3},\sqrt{3})$ et $(-\sqrt{3},-\sqrt{3},-\sqrt{3},-\sqrt{3},-\sqrt{3},-\sqrt{3})$ sont des $\lambda$-quiddités sur $<\sqrt{3}>$. Aussi, on peut les utiliser pour réduire toutes les solutions de \eqref{a} de taille supérieure à 7 dont quatre éléments consécutifs sont égaux et valent $\sqrt{3}$ ou $-\sqrt{3}$. 
\\
\\Donc, les $\lambda$-quiddités sur $<\sqrt{3}>$ de taille supérieure à 7 sont réductibles (par la proposition \ref{32} i) et la discussion précédente). 
\\
\\Soit $(a_{1},\ldots,a_{6})$ une $\lambda$-quiddité sur $<\sqrt{3}>$ de taille 6. Si cette solution contient 0 alors, par la proposition \ref{32} i), elle est réductible. Si cette solution ne contient pas 0 alors, par ce qui précède, elle contient quatre $\sqrt{3}$ (ou $-\sqrt{3}$) consécutifs. Quitte à effectuer des permutations circulaires, la solution est de la forme $(a\sqrt{3},\sqrt{3},\sqrt{3},\sqrt{3},\sqrt{3},b\sqrt{3})$ ou $(-a\sqrt{3},-\sqrt{3},-\sqrt{3},-\sqrt{3},-\sqrt{3},-b\sqrt{3})$. Si $(a\sqrt{3},\sqrt{3},\sqrt{3},\sqrt{3},\sqrt{3},b\sqrt{3})$ est une solution de \eqref{a} alors on a 
\[0=K_{5}(a\sqrt{3},\sqrt{3},\sqrt{3},\sqrt{3},\sqrt{3})=a\sqrt{3}K_{4}(\sqrt{3},\sqrt{3},\sqrt{3},\sqrt{3})-K_{3}(\sqrt{3},\sqrt{3},\sqrt{3})=a\sqrt{3}-\sqrt{3}.\]

\noindent Donc, $a=1$. On montre de même que $b=1$. Le cas $(-a\sqrt{3},-\sqrt{3},-\sqrt{3},-\sqrt{3},-\sqrt{3},-b\sqrt{3})$ solution se traite de façon analogue. De plus, ces deux solutions sont irréductibles car dans le cas contraire elles seraient la somme de deux solutions de taille 4 et contiendraient 0.
\\
\\Ainsi, les $\lambda$-quiddités irréductibles sur $<\sqrt{3}>$ sont celles données dans l'énoncé.

\qed

\subsection{Quelques remarques concernant les $\lambda$-quiddités irréductibles sur $\mathbb{Z}[\sqrt{k}]$}

Les résultats obtenus sur les sous-groupes $<\sqrt{k}>$ incitent naturellement à se tourner vers les anneaux $\mathbb{Z}[\sqrt{k}]$. En particulier, il est légitime de se demander ce que devienne les résultats d'irréductibilité obtenus pour les sous-groupes lorsque l'on se place sur les anneaux.
\\
\\Commençons par remarquer que les $\lambda$-quiddités irréductibles sur $<\sqrt{k}>$ sont toujours des $\lambda$-quiddités irréductibles sur $\mathbb{Z}[\sqrt{k}]$, et même plus généralement sur tous les sous-groupes de $\mathbb{C}$ contenant $<\sqrt{k}>$. En effet, soit $G$ un tel groupe. Les uplets du théorème \ref{25} sont toujours des solutions, puisque le calcul matriciel ne dépend que de leurs composantes. Les solutions de taille 4 sont irréductibles puisqu'elles ne contiennent pas $\pm 1$. Considérons maintenant $(\sqrt{3},\sqrt{3},\sqrt{3},\sqrt{3},\sqrt{3},\sqrt{3})$. Supposons par l'absurde que cette $\lambda$-quiddité est réductible sur $G$. Il existe une solution de la forme $(a,\sqrt{3},\ldots,\sqrt{3},b)$ de taille $3 \leq l \leq 5$, avec $(a,b) \in G^{2}$. En particulier, on a $K_{l-2}(\sqrt{3},\ldots,\sqrt{3})=\pm 1$. Or, $K_{1}(\sqrt{3})=\sqrt{3}$, $K_{2}(\sqrt{3},\sqrt{3})=2$ et $K_{3}(\sqrt{3},\sqrt{3},\sqrt{3})=\sqrt{3}$, ce qui est absurde. Par la proposition \ref{opposé}, $(-\sqrt{3},-\sqrt{3},-\sqrt{3},-\sqrt{3},-\sqrt{3},-\sqrt{3})$ est également irréductible.
\\
\\En revanche, il y a en général d'autres $\lambda$-quiddités irréductibles. Tout d'abord, il faut rajouter, pour $k>1$, $(1,1,1)$, $(-1,-1,-1)$, $(a+b\sqrt{k},0,-a-b\sqrt{k},0)$ et $(0,a+b\sqrt{k},0,-a-b\sqrt{k})$, avec $a+b\sqrt{k} \neq \pm 1$. Si $k$ est un carré alors $\mathbb{Z}[\sqrt{k}]=\mathbb{Z}$ et on connaît toutes les $\lambda$-quiddités irréductibles. Si $k$ n'est pas un carré, il peut y avoir d'autres solutions, qui sont bien sûr intimement liées aux propriétés arithmétiques de l'entier $k$. 
\\
\\Remarquons notamment que les $\lambda$-quiddités de taille 4 sur $\mathbb{Z}[\sqrt{k}]$ sont étroitement liées aux solutions (entières) des équations de Pell-Fermat, qui sont les équations diophantiennes de la forme $x^{2}-ky^{2}=m$. En effet, les solutions de $x^{2}-ky^{2}=2$ permettent de construire les $\lambda$-quiddités 
\[(x-y\sqrt{k},x+y\sqrt{k},x-y\sqrt{k},x+y\sqrt{k}),\] qui sont irréductibles (puisque $(1,0)$ et $(-1,0)$ ne sont pas solutions de notre équation diophantienne). Notons par ailleurs qu'une équation de Pell-Fermat a soit une infinité de solutions soit aucune solution. Par exemple, $x^{2}-3y^{2}=2$ n'a pas de solution, car sinon on aurait, en prenant l'équation modulo 3, $\overline{-1}$ carré dans $\mathbb{Z}/3\mathbb{Z}$. En revanche, $x^{2}-7y^{2}=2$ a une infinité de solutions, par exemple $(3,1)$ ou $(45,17)$. Ainsi, $(45-17\sqrt{7},45+17\sqrt{7},45-17\sqrt{7},45+17\sqrt{7})$ est une $\lambda$-quiddité irréductible sur $\mathbb{Z}[\sqrt{7}]$. Par ailleurs, pour construire une $\lambda$-quiddité de taille 4, on doit chercher des entiers $x,y,z,t$ tels que $(x+y\sqrt{k})(z+t\sqrt{k})=2$. En prenant la norme $N(x+y\sqrt{k})=x^{2}-ky^{2}$, on voit que $x^{2}-ky^{2}=d$ et $z^{2}-kt^{2}=d'$ avec $dd'=4$.
\\
\\Pour certaines valeurs de $k$, on peut trouver assez facilement des $\lambda$-quiddités irréductibles de plus grande taille. Plaçons-nous par exemple sur $\mathbb{Z}[\sqrt{2}]$. Les uplets suivants sont des $\lambda$-quiddités irréductibles :
\begin{itemize}
\item $(1-\sqrt{2},-2-2\sqrt{2},1-\sqrt{2},-2-2\sqrt{2})$, $(-2-\sqrt{2},-2+\sqrt{2},-2-\sqrt{2},-2+\sqrt{2})$;
\item $(1-\sqrt{2},-2-\sqrt{2},-2-\sqrt{2},1-\sqrt{2},-5-4\sqrt{2})$, $(-3+2\sqrt{2},-2-\sqrt{2},-2,1-\sqrt{2},3+2\sqrt{2})$, $(-1-\sqrt{2},-\sqrt{2},-2+\sqrt{2},1+\sqrt{2},-3+2\sqrt{2})$;
\item $(4-3\sqrt{2},-2-2\sqrt{2},-2-2\sqrt{2},-2,-2+\sqrt{2},-20-14\sqrt{2})$, 
\\$(-6+4\sqrt{2},-2-2\sqrt{2},-2-2\sqrt{2},-2+\sqrt{2},-2\sqrt{2},-4-3\sqrt{2})$.
\end{itemize}

\noindent Cette liste est bien entendu non exhaustive. Cela dit, elle est suffisamment étoffée pour constater que le cas de $\mathbb{Z}[\sqrt{k}]$ est infiniment plus complexe que celui de $<\sqrt{k}>$ lorsque $k$ n'est pas un carré. On peut même, sans trop s'avancer, émettre la conjecture suivante : si $k$ n'est pas un carré alors la taille des $\lambda$-quiddités irréductibles sur $\mathbb{Z}[\sqrt{k}]$ n'est pas majorée.

\section{Démonstration du théorème \ref{26}}
\label{preuve26}

\subsection{Preuve des bijections}

Nous allons commencer par établir le point i) du théorème. Pour cela, on commence par le résultat ci-dessous :

\begin{proposition}
\label{51}

Soit $k \in \bN$. Il n'y a pas de $\lambda$-quiddité de taille impaire sur $<i\sqrt{k}>$.

\end{proposition}

\begin{proof}

Soit $n$ un entier positif impair. Supposons qu'il existe une solution $(a_{1},\ldots,a_{n})$ de \eqref{a}, avec, pour tout $j$ dans $[\![1;n]\!]$, $a_{j}=i k_{j}\sqrt{k}$ et $k_{j} \in \mathbb{Z}$. Il existe $\epsilon \in \{\pm 1 \}$ tel que : \[\epsilon Id=M_{n}(a_{1},\ldots,a_{n})=\begin{pmatrix}
    K_{n}(a_{1},\ldots,a_{n}) & -K_{n-1}(a_{2},\ldots,a_{n}) \\
    K_{n-1}(a_{1},\ldots,a_{n-1})  & -K_{n-2}(a_{2},\ldots,a_{n-1}) 
   \end{pmatrix}.\]
\noindent En particulier, $K_{n}(a_{1},\ldots,a_{n})=\epsilon \in \mathbb{R}$.
\\
\\Par l'algorithme d'Euler, $K_{n}(a_{1},\ldots,a_{n})$ est une somme dont chacun des termes est un produit d'un nombre impair de $a_{j}$ (et éventuellement de -1). On regroupe les $i$ présents dans chaque terme. Chacun des termes de la somme est alors le produit de $i$ à une puissance impaire multiplié par des multiples entiers de $\sqrt{k}$. Comme les puissances impaires de $i$ valent $\pm i$, $K_{n}(a_{1},\ldots,a_{n})$ est un imaginaire pur.
\\
\\Ainsi, $K_{n}(a_{1},\ldots,a_{n}) \in \mathbb{R} \cap i\mathbb{R}$, c'est-à-dire $K_{n}(a_{1},\ldots,a_{n})=0$. Ceci est absurde puisque $\epsilon \in \{\pm 1 \}$.
\\
\\Donc, \eqref{a} n'a pas de solution de taille impaire sur $<i\sqrt{k}>$.

\end{proof}

\noindent On s'intéresse pour l'instant à l'équation sur $\mathbb{N}i\sqrt{k}$. On commence par le résultat suivant :

\begin{lemma}
\label{52}

Soit $k \in \bN$. Les solutions de \eqref{a} sur $\mathbb{N}i\sqrt{k}$ contiennent un 0.
	
\end{lemma}	

\begin{proof}

Soit $(a_{1},\ldots,a_{n})$ une solution de \eqref{a} sur $\mathbb{N}i\sqrt{k}$. Par la proposition précédente, $n$ est pair. On suppose que $n$ est divisible par 4. Pour tout $j$ dans $[\![1;n]\!]$, $a_{j}=i k_{j}\sqrt{k}$ avec $k_{j} \in \mathbb{N}$. Il existe $\epsilon \in \{\pm 1 \}$ tel que : \[\epsilon Id=M_{n}(a_{1},\ldots,a_{n})=\begin{pmatrix}
    K_{n}(a_{1},\ldots,a_{n}) & -K_{n-1}(a_{2},\ldots,a_{n}) \\
    K_{n-1}(a_{1},\ldots,a_{n-1})  & -K_{n-2}(a_{2},\ldots,a_{n-1}) 
   \end{pmatrix}.\]
\noindent En particulier, $K_{n}(a_{1},\ldots,a_{n})=\epsilon$.
\\
\\On regroupe les $i$ présents dans chaque terme de la somme définissant le continuant $K_{n}(a_{1},\ldots,a_{n})$ (algorithme d'Euler). Les termes dans lesquels on a supprimé un nombre pair de paire de $a_{j}$ consécutifs sont multipliés par $i\sqrt{k}$ puissance un multiple de 4, c'est-à-dire par une puissance paire de $k$. Les termes dans lesquels on a supprimé un nombre impair de paire de $a_{j}$ consécutifs sont multipliés par -1 et par $i\sqrt{k}$ puissance un entier divisible par 2 mais par 4, c'est-à-dire par $(-1) \times (-1)=1$ multiplié par une puissance de $k$. 
\\
\\Ainsi, $K_{n}(a_{1},\ldots,a_{n})$ est égal à 1 (obtenu en supprimant toutes les paires possibles) plus une somme d'entiers positifs. Donc, nécessairement $K_{n}(a_{1},\ldots,a_{n})=1$ et un des $a_{j}$ est nul (car sinon on aurait $K_{n}(a_{1},\ldots,a_{n}) >1)$.
\\
\\On procède de façon analogue si $n$ n'est pas divisible par 4 (dans ce cas l'opposé de $K_{n}(a_{1},\ldots,a_{n})$ est égal à 1 plus une somme d'entiers positifs).

\end{proof}

\noindent Ces résultats permettent de classifier les $\lambda$-quiddités irréductibles sur $\mathbb{N}i\sqrt{k}$.

\begin{theorem}
\label{53}

Les $\lambda$-quiddités sur $\mathbb{N}i\sqrt{k}$ sont les $2n$-uplets ne contenant que 0. En particulier, $(0,0,0,0)$ est la seule $\lambda$-quiddité irréductible sur $\mathbb{N}i\sqrt{k}$.
	
\end{theorem}	

\begin{proof}

Par la proposition \ref{51}, les solutions de \eqref{a} sur $\mathbb{N}i\sqrt{k}$ sont de taille paire. Supposons par l'absurde qu'il existe des solutions non nulles de \eqref{a} sur $\mathbb{N}i\sqrt{k}$. Soit $(a_{1},\ldots,a_{n})$ une $\lambda$-quiddité sur $\mathbb{N}i\sqrt{k}$ possédant une composante non nulle. Par le lemme précédent, il existe $j \in [\![1;n]\!]$ tel que $a_{j}=0$.
\\
\\Or, on a \[\begin{pmatrix}
    a & -1 \\
    1    & 0 
    \end{pmatrix} \begin{pmatrix}
    0 & -1 \\
    1    & 0 
    \end{pmatrix} \begin{pmatrix}
    b & -1 \\
    1    & 0 
    \end{pmatrix} = \begin{pmatrix}
    -a-b & 1 \\
    -1    & 0 
    \end{pmatrix}= -\begin{pmatrix}
    a+b & -1 \\
    1    & 0 
    \end{pmatrix}.\]
\noindent En utilisant cela, $(a_{1},\ldots,a_{j-2},a_{j-1}+a_{j+1},a_{j+2},\ldots,a_{n})$ est une solution de \eqref{a} sur $\mathbb{N}i\sqrt{k}$. Cette solution possède donc un zéro et une composante non nulle et on peut ainsi réitérer le processus.
\\
\\En procédant ainsi, on arrive à une solution de taille 2 possédant une composante non nulle. Ceci est absurde par la proposition \ref{31}.
\\
\\Donc, les $\lambda$-quiddités sur $\mathbb{N}i\sqrt{k}$ sont des $2n$-uplets ne contenant que des 0. De plus, tous les $2n$-uplets ne contenant que des 0 sont des $\lambda$-quiddités sur $\mathbb{N}i\sqrt{k}$. En particulier, $(0,0,0,0)$ est la seule $\lambda$-quiddité irréductible sur $\mathbb{N}i\sqrt{k}$ (proposition \ref{32} i)).

\end{proof}

\noindent On revient maintenant au cas de $\mathbb{Z}i\sqrt{k}$ :

\begin{proof}[Démonstration du théorème \ref{26}]

Soit $k \in \bN$. $\Omega_{k}$ et $\Xi_{k}$ n'ont que des éléments de taille paire (proposition \ref{51} et section précédente). Soit $(a_{1}i,a_{2}i,\ldots,a_{2n-1}i,a_{2n}i)$ une $\lambda$-quiddité sur $<i\sqrt{k}>$. Il existe $\epsilon \in \{\pm 1 \}$ tel que :
\begin{eqnarray*}
\epsilon Id &=& M_{2n}(a_{1}i,a_{2}i,\ldots,a_{2n-1}i,a_{2n}i) \\
            &=& \begin{pmatrix}
    K_{2n}(a_{1}i,a_{2}i,\ldots,a_{2n-1}i,a_{2n}i) & -K_{2n-1}(a_{2}i,\ldots,a_{2n-1}i,a_{2n}i) \\
    K_{2n-1}(a_{1}i,a_{2}i,\ldots,a_{2n-1}i)  & -K_{2n-2}(a_{2}i,\ldots,a_{2n-1}i) 
   \end{pmatrix}.\\
\end{eqnarray*}
\noindent Par l'algorithme d'Euler, $K_{2n}(a_{1}i,a_{2}i,\ldots,a_{2n-1}i,a_{2n}i)$ est une somme de produits des $a_{j}i$. Dans chaque terme d'une de ces sommes, pour chaque $j$ impair $a_{j}$ est suivi d'un $a_{k}$ avec $k$ pair. On fait alors la manipulation suivante :\[a_{j}ia_{k}i=a_{j}a_{k}i^{2}=a_{j}(-a_{k}).\]
Avec cette manipulation, on a \[K_{2n}(a_{1}i,a_{2}i,\ldots,a_{2n-1}i,a_{2n}i)=K_{2n}(a_{1},-a_{2},\ldots,a_{2n-1},-a_{2n}).\]  
\noindent On procède de façon analogue pour $K_{2n-2}(a_{2}i,\ldots,a_{2n-1}i)$ (puisque pour pour chaque $j$ impair $a_{j}$ est précédé d'un $a_{l}$ avec $l$ pair). $K_{2n-1}(a_{1}i,a_{2}i,\ldots,a_{2n-1}i)$ est une somme de produits des $a_{j}i$. Pour chaque terme de la somme, on a deux choix :
\begin{itemize}
\item le terme est de la forme $a_{j}i$ avec $j$ impair;
\item le terme est un produit de $2u+1$ éléments de la forme $a_{j}i$ avec $u+1$ indices $j$ impair. Pour chaque $m$ pair $a_{m}$ est précédé d'un $a_{l}$ avec $l$ impair. On fait alors la manipulation suivante :\[a_{l}ia_{m}i=a_{l}a_{m}i^{2}=a_{l}(-a_{m}).\]
\end{itemize}
Cela donne $0=K_{2n-1}(a_{1}i,a_{2}i,\ldots,a_{2n-1}i)=iK_{2n-1}(a_{1},-a_{2},\ldots,a_{2n-1})$. En particulier, on a \[0=K_{2n-1}(a_{1},-a_{2},\ldots,a_{2n-1}).\] 

\noindent On procède de façon analogue pour $K_{2n-1}(a_{2}i,a_{3}i,\ldots,a_{2n}i)$.
\\
\\Donc, $(a_{1},-a_{2},\ldots,a_{2n-1},-a_{2n})$ est une $\lambda$-quiddité sur $<\sqrt{k}>$. Ainsi, $\varphi_{k}$ est bien défini.
\\
\\ Comme $\varphi_{k}$ est injective, il nous reste à montrer que $\varphi_{k}$ est surjective.
\\
\\ Soit $(a_{1},a_{2},\ldots,a_{2n-1},a_{2n})$ une $\lambda$-quiddité de taille paire sur $<\sqrt{k}>$. Par les mêmes arguments que ceux utilisés précédemment, $(a_{1}i,-a_{2}i,\ldots,a_{2n-1}i,-a_{2n}i)$ est une $\lambda$-quiddité sur $<i\sqrt{k}>$ et son image par $\varphi_{k}$ est $(a_{1},a_{2},\ldots,a_{2n-1},a_{2n})$. Donc, $\varphi_{k}$ est surjective.
\\
\\Ainsi, $\varphi_{k}$ établit une bijection entre $\Omega_{k}$ et $\Xi_{k}$.
\\
\\ On considère maintenant l'irréductibilité des solutions de \eqref{a} sur $<i\sqrt{k}>$ lorsque $k \neq 1$.
\\
\\Soit $(a_{1}i,a_{2}i,\ldots,a_{2n-1}i,a_{2n}i)$ une $\lambda$-quiddité sur $<i\sqrt{k}>$. Supposons que $\varphi_{k}((a_{1}i,a_{2}i,\ldots,a_{2n-1}i,a_{2n}i))$ est réductible. Comme $k \neq 1$, il existe $(b_{1},b_{2},\ldots,b_{2l-1},b_{2l})$ et $(c_{1},c_{2},\ldots,c_{2l'-1},c_{2l'})$ deux $\lambda$-quiddités de taille paire sur $<\sqrt{k}>$ de taille supérieure à 3 telles que :

\begin{eqnarray*}
(a_{1},-a_{2},\ldots,a_{2n-1},-a_{2n}) &=& \varphi_{k}((a_{1}i,a_{2}i,\ldots,a_{2n-1}i,a_{2n}i)) \\
                                       &\sim& (b_{1},b_{2},\ldots,b_{2l-1},b_{2l}) \oplus (c_{1},c_{2},\ldots,c_{2l'-1},c_{2l'}).\\
\end{eqnarray*}

\noindent Ainsi, $(a_{1}i,a_{2}i,\ldots,a_{2n-1}i,a_{2n}i) \sim (b_{1}i,-b_{2}i,\ldots,b_{2l-1}i,-b_{2l}i) \oplus (-c_{1}i,c_{2}i,\ldots,-c_{2l'-1}i,c_{2l'}i)$. Or, on a montré que $(b_{1}i,-b_{2}i,\ldots,b_{2l-1}i,-b_{2l}i)$ et $(c_{1}i,-c_{2}i,\ldots,c_{2l'-1}i,-c_{2l'}i)$ sont des $\lambda$-quiddités sur $<i\sqrt{k}>$. Par la proposition \ref{opposé}, $(-c_{1}i,c_{2}i,\ldots,-c_{2l'-1}i,c_{2l'}i)$ est une $\lambda$-quiddité sur $<i\sqrt{k}>$. Donc, l'image par $\varphi_{k}$ d'une solution irréductible sur $<i\sqrt{k}>$ est une solution irréductible sur $<\sqrt{k}>$.
\\
\\Soit $(a_{1}i,a_{2}i,\ldots,a_{2n-1}i,a_{2n}i)$ une $\lambda$-quiddité réductible sur $<i\sqrt{k}>$. Il existe
$(b_{1}i,b_{2}i,\ldots,b_{2l-1}i,b_{2l}i)$ et $(c_{1}i,c_{2}i,\ldots,c_{2l'-1}i,c_{2l'}i)$ deux $\lambda$-quiddités sur $<i\sqrt{k}>$ de taille supérieure à 3 telles que
\[(a_{1}i,a_{2}i,\ldots,a_{2n-1}i,a_{2n}i) \sim (b_{1}i,b_{2}i,\ldots,b_{2l-1}i,b_{2l}i) \oplus (c_{1}i,c_{2}i,\ldots,c_{2l'-1}i,c_{2l'}i).\]

\noindent On a 
\begin{eqnarray*}
\varphi_{k}((a_{1}i,a_{2}i,\ldots,a_{2n-1}i,a_{2n}i)) &=& (a_{1},-a_{2},\ldots,a_{2n-1},-a_{2n}) \\
                                                  &\sim& (b_{1},-b_{2},\ldots,b_{2l-1},-b_{2l}) \oplus (-c_{1},c_{2},\ldots,-c_{2l'-1},c_{2l'}). \\
\end{eqnarray*}

\noindent Or, $(c_{1},-c_{2},\ldots,c_{2l'-1},-c_{2l'})=\varphi_{k}((c_{1}i,c_{2}i,\ldots,c_{2l'-1}i,c_{2l'}i))$. Donc, $(c_{1},-c_{2},\ldots,c_{2l'-1},-c_{2l'})$ est une solution de \eqref{a} sur $<\sqrt{k}>$. Par la proposition \ref{opposé}, $(-c_{1},c_{2},\ldots,-c_{2l'-1},c_{2l'})$ est une solution de \eqref{a} sur $<\sqrt{k}>$. Donc, $\varphi_{k}((a_{1}i,a_{2}i,\ldots,a_{2n-1}i,a_{2n}i))$ est réductible. Ainsi, $\varphi_{k}$ établit une surjection entre les solutions irréductibles de $\Omega_{k}$ et les solutions irréductibles de $\Xi_{k}$.
\\
\\Donc, $\varphi_{k}$ établit une bijection entre les solutions irréductibles sur $<i\sqrt{k}>$ et les solutions irréductibles sur $<\sqrt{k}>$.

\end{proof}

Malheureusement, $\varphi_{1}$ n'établit pas une bijection entre les solutions irréductibles sur $<i>$ et les solutions irréductibles sur $\bZ$. En effet, il existe des solutions irréductibles de taille 3 sur $\bZ$ (voir Théorème \ref{35}) alors que cela n'est pas le cas sur $<i>$ (voir proposition \ref{51}). 

\subsection{Solutions pairement irréductibles}

Nous allons maintenant donner quelques éléments sur l'irréductibilité des $\lambda$-quiddités sur $\mathbb{Z}i$. On commence par la définition suivante qui concerne les solutions de taille paire sur $\bZ$.

\begin{definition}
\label{44}

Une $\lambda$-quiddité sur $\mathbb{Z}$ de taille paire est dite pairement réductible si on peut l'écrire comme une somme de deux $\lambda$-quiddités sur $\bZ$ de taille paire supérieure à 4. Dans le cas contraire, la $\lambda$-quiddité est dite pairement irréductible.

\end{definition}

\begin{examples} 
{\rm
\begin{itemize}
\item Les $\lambda$-quiddités de taille 4 sont pairement irréductibles;
\item $(2,2,1,4,1,2)$ est pairement réductible car $(2,2,1,4,1,2)=(1,2,1,2) \oplus (2,1,2,1)$;
\item $(1,2,1,2,1,2,1,2)$ est pairement réductible car $(1,2,1,2,1,2,1,2)=(1,2,1,2) \oplus (0,1,2,1,2,0)$;
\item $(1,1,1,1,1,1)$ est pairement irréductible car dans le cas contraire elle serait la somme de deux $\lambda$-quiddités de taille 4 et il n'y a pas de $\lambda$-quiddité de la forme $(a,1,1,b)$.
\end{itemize}
}
\end{examples}

\noindent Graphiquement, la notion de $\lambda$-quiddité sur $\mathbb{Z}$ pairement réductible signifie qu'il existe un étiquetage admissible aboutissant à cette quiddité et possédant les deux propriétés suivantes :
\begin{itemize}
\item on peut, en coupant suivant une diagonale, obtenir deux sous-polygones possédant un nombre pair de sommet;
\item l'étiquetage admissible du polygone initial induit un étiquetage admissible pour chacun des deux sous-polygones.
\end{itemize}

\noindent Par exemple, cela donne (avec tous les triangles étiquetés avec 1) :

$$
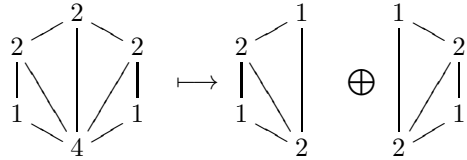
\shorthandoff{; :!?}
\xymatrix @!0 @R=0.45cm @C=0.8cm
{
&2\ar@{-}[ld]\ar@{-}[rd]&
\\
2\ar@{-}[dd]\ar@{-}[]&&2\ar@{-}[dd]
\\
 &&& \longmapsto
\\
1\ar@{-}[]&&1\ar@{-}[]
\\
&4\ar@{-}[lu]\ar@{-}[ru]\ar@{-}[uuuu]\ar@{-}[ruuu]\ar@{-}[luuu]&
}
\xymatrix @!0 @R=0.45cm @C=0.8cm
{
&1\ar@{-}[ld]&
\\
2\ar@{-}[dd]
\\
&&\bigoplus
\\
1\ar@{-}[]&&
\\
&2\ar@{-}[lu]\ar@{-}[uuuu]\ar@{-}[luuu]&
} 
\xymatrix @!0 @R=0.45cm @C=0.8cm
{
1\ar@{-}[rd]&
\\
&2\ar@{-}[dd]
\\
\\
&1\ar@{-}[]
\\
2\ar@{-}[ru]\ar@{-}[uuuu]\ar@{-}[ruuu]
}$$

\noindent Cette interprétation graphique doit néanmoins être considérée avec précaution. En effet, il est possible pour une même quiddité d'avoir deux étiquetages admissibles différents, un ne vérifiant pas les deux conditions précédentes et un autre les vérifiant. Par exemple, pour $(1,2,1,2,1,2,1,2)$ on a :

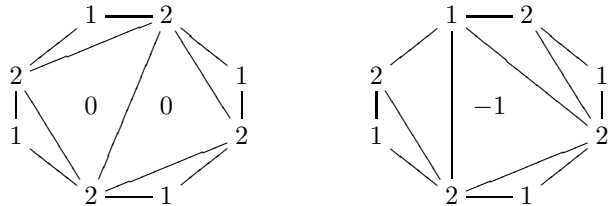
$$
\shorthandoff{; :!?}
\xymatrix @!0 @R=0.40cm @C=0.5cm
{
&&1\ar@{-}[lldd] \ar@{-}[rr]&&2\ar@{-}[rrdd]\ar@{-}[lldddddd]\ar@{-}[rrdddd]&
\\
&&&
\\
2\ar@{-}[dd]\ar@{-}[rrrruu]&&&&&& 1 \ar@{-}[dd]
\\
&&0&&0
\\
1&&&&&& 2
\\
&&&
\\
&&2 \ar@{-}[rr]\ar@{-}[rrrruu]\ar@{-}[lluuuu]\ar@{-}[lluu] &&1 \ar@{-}[rruu]
}
\qquad
\qquad
\xymatrix @!0 @R=0.40cm @C=0.5cm
{
&&1\ar@{-}[lldd] \ar@{-}[dddddd]\ar@{-}[rrrrdddd]\ar@{-}[rr]&&2\ar@{-}[rrdd]\ar@{-}[rrdddd]&
\\
&&&
\\
2\ar@{-}[dd]&&&&&& 1 \ar@{-}[dd]
\\
&&&-1
\\
1&&&&&& 2
\\
&&&
\\
&&2 \ar@{-}[rr]\ar@{-}[rrrruu]\ar@{-}[lluuuu]\ar@{-}[lluu] &&1 \ar@{-}[rruu]
}
$$

\noindent Les triangles non étiquetés sur la figure sont étiquetés avec 1.
\\
\\Cette notion de solution pairement réductible est reliée à notre problème initial via le résultat ci-dessous :

\begin{proposition}
\label{52}

$(a_{1}i,a_{2}i,\ldots,a_{2n-1}i,a_{2n}i)$ est une $\lambda$-quiddité irréductible sur $<i>$ si et seulement si $\varphi_{1}(a_{1}i,a_{2}i,\ldots,a_{2n-1}i,a_{2n}i)=(a_{1},-a_{2},\ldots,a_{2n-1},-a_{2n})$ est pairement irréductible.

\end{proposition}

\begin{proof}

La preuve est analogue à celle du théorème \ref{26}. 

\end{proof}

\noindent On termine par la conjecture ci-dessous :

\begin{con}

Il existe des solutions pairement irréductibles de taille arbitrairement grande.

\end{con}


\begin{thebibliography}{1}

\bibitem{BR}
F. Bergeron, C. Reutenauer, 
{\it $SL_{k}$-tilings of the plane}, 
Illinois J. Math., Vol. 54 no. 1, (2010), pp 263-300.

\bibitem{Ca}
G. Cantor, 
{\it Sur une propriété du système de tous les nombres algébriques réels}, 
Acta Mathemetica, Vol. 2, (1883), pp 305-310.

\bibitem{CO}
C. Conley, V. Ovsienko, 
{\it Rotundus: triangulations, Chebyshev polynomials, and Pfaffians}, 
Math. Intelligencer, Vol. 40 no. 3, (2018), pp 45-50.

\bibitem{C} 
M. Cuntz,
{\it A combinatorial model for tame frieze patterns}, 
Munster J. Math., Vol. 12 no. 1, (2019), pp 49-56.

\bibitem{CH} 
M. Cuntz, T. Holm,
{\it Frieze patterns over integers and other subsets of the complex numbers}, 
J. Comb. Algebra., Vol. 3 no. 2, (2019), pp 153-188.

\bibitem{Ge}
A. O. Gelfond,
{\it Sur le septième Problème de D. Hilbert},
Bulletin de l'Académie des Sciences de l'URSS. Classe des sciences mathématiques et na, Vol. 4, (1934), pp 623-634.

\bibitem{G} 
I. Gozard,
{\it Théorie de Galois - niveau L3-M1 - 2e édition}, 
Ellipses, 2009.

\bibitem{H}
C. Hermite,
{\it Sur la fonction exponentielle},
Comptes rendus hebdomadaires des séances de l'Académie des sciences, Vol. 77, (1873), pp 18-24, 74-79, 226-233 et 285-293.

\bibitem{L}
J. Liouville,
À propos de l'existence des nombres transcendants, 
Compte rendu des séances de l’Académie des sciences. Séance du lundi 13 mai 1844. Présidence de M. Elie de Beaumont. Disponible en ligne à l’adresse : http://www.bibnum.education.fr/mathematiques/theorie-des-nombres/propos-de-l-existence-des-nombres-transcendants

\bibitem{M1}
F. Mabilat,
\textit{Combinatoire des sous-groupes de congruence du groupe modulaire}, 
Annales Mathématiques Blaise Pascal, Vol. 28 no. 1, (2021), pp. 7-43. doi : 10.5802/ambp.398. https://ambp.centre-mersenne.org/articles/10.5802/ambp.398/.

\bibitem{M2}
F. Mabilat,
\textit{$\lambda$-quiddité sur $\mathbb{Z}[\alpha]$ avec $\alpha$ transcendant},
Mathematica Scandinavica, Vol. 128 no. 1, (2022), pp 5-13, https://doi.org/10.7146/math.scand.a-128972.

\bibitem{M3}
F. Mabilat,
{\it Combinatoire des sous-groupes de congruence du groupe modulaire II}. 
Annales Mathématiques Blaise Pascal, Vol. 28 no. 2, (2021), pp. 199-229. doi : 10.5802/ambp.404. https://ambp.centre-mersenne.org/articles/10.5802/ambp.404/.

\bibitem{M4}
F. Mabilat,
\textit{Entiers monomialement irréductibles}, hal-03487145, arXiv:2112.10410.

\bibitem{M5}
F. Mabilat,
\textit{Solutions monomiales minimales irréductibles dans $SL_{2}(\mathbb{Z}/p^{n}\mathbb{Z})$},
Bulletin des Sciences Mathématiques, Vol. 194, Article 103456, (2024), https://doi.org/10.1016/j.bulsci.2024.103456, hal-03573421, arxiv:2202.07279.

\bibitem{Mo1}
S. Morier-Genoud,
\textit{Coxeter's frieze patterns at the crossroad of algebra, geometry and combinatorics}, 
Bull. Lond. Math. Soc., Vol. 47 no. 6, (2015), pp 895-938.

\bibitem{Mo2}
S. Morier-Genoud,
\textit{Counting Coxeter's friezes over a finite field via moduli spaces}, 
Algebraic combinatoric, Vol. 4 no. 2, (2021), pp 225-240.

 \bibitem{MO}
S. Morier-Genoud, V. Ovsienko,
\textit{Farey Boat. Continued fractions and triangulations,
modular group and polygon dissections }, 
Jahresber. Dtsch. Math. Ver., Vol. 121 no. 2, (2019), pp 91-136, https://doi.org/10.1365/s13291-019-00197-7.

\bibitem{O} 
V. Ovsienko, 
{\it Partitions of unity in $SL(2,\mathbb{Z})$,  negative continued fractions,  and dissections of polygons,} 
Res. Math. Sci., Vol. 5 no. 2, (2018), Article 21, 25 pp. 

\bibitem{WZ}
M. Weber, M. Zhao,
{\it Factorization of frieze patterns,}
Revista de la Unión Matemática Argentina, Vol. 60 no. 2, (2019), pp 407-415.



\end{thebibliography}
\end{document}